\documentclass[11pt]{amsart}
\usepackage{amssymb}
\usepackage{amsmath}
\usepackage{amsthm}
\usepackage{color}
\usepackage{appendix}
\usepackage[unicode,breaklinks=true,colorlinks=true]{hyperref}
\usepackage{CJKutf8}

\makeatletter
\def\LaTeX{\leavevmode L\raise.42ex
	\hbox{\kern-.3em\size{\sf@size}{0pt}\selectfont A}\kern-.15em\TeX}
\makeatother

\newcommand{\BibTeX}{{\rm B\kern-.05em{\sc
			i\kern-.025emb}\kern-.08em\TeX}}

\makeatletter
\def\@currentlabel{2.1}\label{e:dispaa}
\def\@currentlabel{2.21}\label{e:dispau}
\def\@currentlabel{2.22}\label{e:dispav}
\def\@currentlabel{2.23}\label{e:dispaw}
\def\@currentlabel{2.24}\label{e:dispax}
\def\theequation{\thesection.\@arabic\c@equation}
\makeatother

\newcounter{mnotecount}[section]

\newcommand{\rmnote}[1]{}

\numberwithin{equation}{section}
\renewcommand{\theequation}{\arabic{section}.\arabic{equation}}
\newtheorem{thm}{Theorem}[section]
\newtheorem{lem}[thm]{Lemma}

\newtheorem{prop}[thm]{Proposition}
\theoremstyle{definition}
\newtheorem{defn}{Definition}[section]
\newtheorem{rem}{Remark}[section]
\newcommand{\M}{\mathbb{M}}
\newcommand{\R}{\mathbb{R}}

\newcommand\supp{\operatorname{supp}}

\allowdisplaybreaks

\begin{document}
		\begin{CJK}{UTF8}{gbsn}
	\author[Y. Yang]{Yifan Yang}
	\address{School of Mathematics and statistics, Xi'an Jiaotong University, Xi'an 710049, P.R. China}
	\email{yangyifan2021@stu.xjtu.edu.cn}
\title[self-similar solutions of the MHD-boussinesq system]{Forward self-similar solutions of the MHD-boussinesq system with newtonian gravitational field}

\begin{abstract}
	This paper is concerned with the existence of forward self-similar solutions to the three-dimensional Magnetohydrodynamic-Boussinesq (MHD-Boussinesq) system with newtonian gravitational field. By employing a blow-up argument and the Leray-Schauder theorem, we construct a forward self-similar solution to this system without imposing any smallness assumptions on the self-similar initial value.
\end{abstract}

 \maketitle
 
\setcounter{equation}{0}
\setcounter{equation}{0}
\section{Introduction}
In this paper, we study the existence of forward self-similar solutions to the Cauchy problem for the MHD-Boussinesq system.
This system integrates the Boussinesq equations of fluid dynamics with Maxwell's equations of electromagnetism, with the displacement current neglected. It models the convection of incompressible flow driven by buoyancy effects from thermal or density fields, along with the Lorentz force generated by the fluid's magnetic field. Specifically, the MHD-Boussinesq system is closely related to a natural type of Rayleigh-B\'{e}nard convection, occurring in a horizontal layer of electrically conducting fluid heated from below in the presence of a magnetic field. For more details on the physical background can be found in \cite{KAG,LLD,MA,MGR1}. The fundamental equations of MHD-Boussinesq is given by
\begin{align}\label{E1.1}
\begin{split}
\left.
\begin{aligned}
\partial_{t} u-\nu\Delta u+(u\cdot\nabla)u+\nabla p=(b\cdot\nabla)b+\theta\nabla G\\
\partial_{t} b-\mu\Delta b+(u\cdot\nabla)b=(b\cdot\nabla)u\\
\partial_{t} \theta-\kappa\Delta\theta+(u\cdot\nabla)\theta=0\\
\mbox{div}\,u =\mbox{div}\,b=0\\
\end{aligned}\ \right\}\ \ \mbox{in}\ \ \ \R^{3}\times (0,+\infty)
\end{split}
\end{align}
with initial condition
\begin{equation}\label{E1.1-initial}
u(\cdot,0)=u_{0},\ \ b(\cdot,0)=b_{0},\ \ \theta(\cdot,0)=\theta_{0}\quad\mbox{on}\quad\R^3,
\end{equation}
where $u:\R^{3}\times\R^{+}\rightarrow\R^{3}$ is the unknown velocity field, $b:\R^{3}\times\R^{+}\rightarrow\R^{3}$ is the magnetic field and the unknown scalar function $\theta:\R^{3}\times\R^{+}\rightarrow\R$ stands for the temperature. The unknown scalar function $p:\R^{3}\times\R^{+}\rightarrow\R$  represents the pressure of the fluid. The positive constant $\nu,\mu,\kappa$ represent the kinematic viscosity, magnetic diffusivity and  heat
conductivity coefficient, respectively. As these physical constants do not play any role in our proof, we assume without loss of generality that $\nu=\mu=\kappa=1$ throughout this paper.

In the first equation of \eqref{E1.1}, $\nabla G$ represents the gravitational force acting on the fluid. The case $\nabla G = ge_{n}$, where $g$ is the gravitational acceleration and $e_{n}$ is the unit vector in the $x_{n}$ direction, has been widely studied (see, e.g., \cite{LP,LY,PX,ZX}). However, in Boussinesq system, i.e. $b=0$ in the MHD-Boussinesq equations \eqref{E1.1}, from some physical point of view, when the region is whole space, it is better to consider the gravitational potential\cite{LBM,EFMS} as follow
$$
G(x)=\int_{\R^{3}}\frac{1}{|x-y|}m(y)dy, 
$$
where $m\geq0$ denotes the mass density of the object acting on the fluid by means of gravitation. Moreover, if the size of the object is negligible, we may choose $G(x)=1/|x|$. In this paper, we consider the singular gravitational force in MHD-Boussinesq equation, that is, taking
\begin{equation}\label{G}
\nabla G(x)=\nabla\frac{1}{|x|}
\end{equation}
in equations \eqref{E1.1}.

Let us first recall the scaling property of the fundamental fluid dynamics equation, the Navier-Stokes equations. In system \eqref{E1.1}, setting $b=0$ and $\theta=0$ yields the incompressible Navier-Stokes equation
\begin{equation}\label{NS-01}
	\left.
	\begin{aligned}
		\partial_{t} u-\Delta u+(u\cdot\nabla)u+\nabla p=0\\
		\mbox{div}\, u=0
	\end{aligned}\ \right\}\ \mbox{in} \ \ \R^{3}\times (0,\infty)
\end{equation}
with initial condition
\begin{equation*}\label{E1.1-initial2}
	u(\cdot,0)=u_{0}(x)\ \ \mbox{in}\ \ \R^3.
\end{equation*}
The system \eqref{NS-01} exhibits the following scaling property: for any $\lambda>0$, if $(u,p)$ is a solution of \eqref{NS-01}, then $(u_{\lambda},p_{\lambda})$ is also a solution, where
\begin{equation*}\label{scaling0}
		u_{\lambda}(x,t)=\lambda u(\lambda x,\lambda^{2}t),\ \
		p_{\lambda}(x,t)=\lambda^{2} p(\lambda x,\lambda^{2}t).
\end{equation*}
A solution $(u,p)$ to \eqref{NS-01} is called forward self-similar if
\begin{equation*}
	u=u_{\lambda},\
	p=p_{\lambda},\ \mbox{for every}\ \lambda>0.
\end{equation*}
In the study of self-similar solution, it is natural to assume that the initial data is also invariant under the scaling
\begin{equation*}
	u_{0}(x)=\lambda u_{0}(\lambda x),
\end{equation*}
which means that the initial conditional must be (-1)-homogeneous.

Similar to the Navier-Stokes equations, the MHD-Boussinesq equations with singular gravitational force \eqref{G} enjoys the following scaling property: for any $\lambda>0$, if $(u,b,\theta,p)$ is a solution of system \eqref{E1.1}, then
\begin{equation}\label{scaling1}
\begin{cases}
u_{\lambda}(x,t)=\lambda u(\lambda x,\lambda^{2}t),\\
b_{\lambda}(x,t)=\lambda b(\lambda x,\lambda^{2}t),\\
\theta_{\lambda}(x,t)=\lambda \theta(\lambda x,\lambda^{2}t),\\
p_{\lambda}(x,t)=\lambda^{2} p(\lambda x,\lambda^{2}t),
\end{cases}
\end{equation}
is also a solution of system \eqref{E1.1}. The forward self-similar solution of \eqref{E1.1} is a solution which is scale-invariant under the scaling \eqref{scaling1}, i.e.,
\begin{equation}\label{scaling11}
	u=u_{\lambda},\ b=b_{\lambda},\ \theta=\theta_{\lambda},\ 
	p=p_{\lambda},\ \mbox{for every}\ \lambda>0.
\end{equation}
And the corresponding initial value \eqref{E1.1-initial} satisfies
\begin{equation}\label{initial scale}
	u_{0}(x)=\lambda u_{0}(\lambda x),\ 	b_{0}(x)=\lambda b_{0}(\lambda x),\ 	\theta_{0}(x)=\lambda \theta_{0}(\lambda x).\ 
\end{equation}

Let us briefly review some related results in fluid dynamics equation. For the Navier-Stokes equations, the existence and uniqueness of forward self-similar solutions have been established under various frameworks, provided suitable smallness conditions are imposed on the initial data. These include Morrey spaces \cite{GT}, Lorentz spaces \cite{B}, and other function spaces \cite{C1,C2,KT2}. However, the existence of forward self-similar solutions for large initial values has been an open problem for a long time. Recently, Jia and \v{S}ver\'{a}k \cite{JS} established the existence of smooth forward self-similar solutions for arbitrary scale-invariant initial values $u_{0} \in C_{loc}^{\alpha}(\R^{3} \setminus {0})$. They achieved this by demonstrating local H\"{o}lder regularity of Leray solutions near $t=0$ using the iteration method of Caffarelli, Kohn, and Nirenberg \cite{CKN}, and then proving existence through the Leray-Schauder fixed point theorem. Their method imposes no restriction on the size of the initial value, though it does not address uniqueness. Later, Korobkov and Tsai \cite{KT} introduced a new strategy to construct large forward self-similar solutions in half-space. Specifically, they employed a blow-up argument to derive $H^{1}$ estimates for the corresponding elliptic system and then used the fixed point theorem to establish the existence of solutions. This blow-up argument was first used by De Giorgi in his work on minimal surfaces \cite{GE} and was more recently applied by Abe and Giga \cite{AG} to obtain the analyticity of Stokes semigroups in $L^{\infty}$-type spaces. For more works on the existence and regularity of self-similar solutions to the Navier-Stokes equations, we refer to the interested
readers to \cite{CGNP,LMZ,NRS,T} and reference therein.

If the fluid is unaffected by temperature, i.e., $\theta=0$ in \eqref{E1.1}, the classical MHD equations are obtained. Due to the highly similar structure of the MHD and Navier-Stokes equations, many analogous results have been obtained. For example, He and Xin \cite{HX} used the Picard contraction principle to construct a class of globally unique forward self-similar solutions with small initial data in Besov space, Lorentz space and pseudo-measure space. Subsequently, the uniqueness result of the self-similar solution was improved in \cite{BRT}. Recently, the existence of self-similar solutions for large initial data has been established in various frameworks, including weak-$L^{3}$ space\cite{L}, Besov space\cite{ZZ} and weighted $L^{2}$ space\cite{FDJ}. More recently, motivated by Korobkov and Tsai \cite{KT}, Yang \cite{Y} constructed a global-time forward self-similar solutions to the MHD equations for arbitrarily large self-similar initial data belong to $L_{loc}^{\infty}(\R^{3}\setminus\{0\})$. On the other hand, when the fluid is not affected by the Lorentz force, that is, $b=0$ in \eqref{E1.1}, one obtains the Boussinesq equations. Results concerning the existence and uniqueness of forward self-similar solutions to the Boussinesq equations with small initial data can be found in \cite{DAF, FVR, GKNP}. However, for large scale-invariant initial values belong to $L^{\infty}_{loc}(\R^{3}\backslash\{0\})$, only \cite{LBGK} has established the existence of self-similar solutions with singular gravitational force \eqref{G}.

The study of equation \eqref{E1.1} originated from investigations into the magnetic B\'{e}nard problem\cite{BHA}. The stability and instability of this problem have attracted considerable attention, leading to a series of results (see, e.g., \cite{GP,MGR,MGR1}). Additionally, the existence and well-posedness of weak solutions to the MHD-Boussinesq equations \eqref{E1.1} with $\nabla G = ge_{n}$ have been recently addressed in \cite{BDG,BDL,CD,LP}. However, to the best of our knowledge, research on self-similar solutions of the MHD-Boussinesq equations \eqref{E1.1} with the singular gravitational force \eqref{G} remains sparse. 

In this paper, we construct a forward self-similar solution to the MHD-Boussinesq equations \eqref{E1.1} for any initial data that are (-1)-homogeneous and essentially bounded on the unit sphere of $\mathbb{R}^{3}$. This work is inspired by \cite{LBGK} and \cite{KT}, and the main result is as follows.
\begin{thm}\label{T1.1}
	Let $(u_{0},b_{0},\theta_{0})\in \mathcal{L}_{loc}^{\infty}(\R^{3}\backslash\{0\})$ be (-1)-homogeneous, with $\mathrm{div} u_{0}=\mathrm{div} b_{0}=0$. Then, there exists at least one forward self-similar solution $(u, b, \theta)$ to problem \eqref{E1.1}-\eqref{E1.1-initial}, which possesses the following properties:
	\begin{itemize}
		\item  $(u,b,\theta)\in BC_{w}\big([0,\infty),\mathcal{L}^{3,\infty}(\R^{3})\big)$;
		\item for all $t>0$ and $2\leq p\leq6$, we have
		\begin{align}\label{T1}
			\begin{split}
				\|u(t)-e^{t\Delta}u_{0}\|_{p}+\|b(t)-e^{t\Delta}b_{0}\|_{p}+\|\theta(t)-e^{t\Delta}\theta_{0}\|_{p}\leq Ct^{\frac{3}{2p}-\frac{1}{2}},\\
				\|\nabla u(t)-\nabla e^{t\Delta}u_{0}\|_{2}+\|\nabla b(t)-\nabla e^{t\Delta}b_{0}\|_{2}+\|\nabla\theta(t)-\nabla e^{t\Delta}\theta_{0}\|_{2}\leq Ct^{-\frac{1}{4}}.
			\end{split}
		\end{align}
	\end{itemize}
\end{thm}
\begin{rem}
	In Theorem \ref{T1.1}, we did not impose any restrictions on the size of the initial data, but correspondingly, we only discussed the existence of solutions without addressing uniqueness. It is worth noting that Theorem \ref{T1.1} yields a global-time forward self-similar solution, because $(u,b,\theta)\in BC_{w}([0,\infty),\mathcal{L}^{3,\infty}(\R^{3}))$ ensures that the constructed solution satisfies the initial condition in a weak sense.
\end{rem}
We will follow Leray's method to prove Theorem \ref{T1.1}, let us outline the strategy of the proof. In fact, the forward self-similar solutions of \eqref{E1.1} are determined by their value at any time moment $t>0$. For instance, choosing $\lambda=1/\sqrt{2t}$ in \eqref{scaling11} allows us to rewrite the forward self-similar solution as
\begin{equation*}\label{scaling2}
		\left\{
	\begin{aligned}
		u(x,t)&=\frac{1}{\sqrt{2t}}U\Big(\frac{x}{\sqrt{2t}}\Big),\\
		b(x,t)&=\frac{1}{\sqrt{2t}}B\Big(\frac{x}{\sqrt{2t}}\Big),\\\theta(x,t)&=\frac{1}{\sqrt{2t}}\Theta\Big(\frac{x}{\sqrt{2t}}\Big),\\
		p(x,t)&=\frac{1}{2t}P\Big(\frac{x}{\sqrt{2t}}\Big),
	\end{aligned}
	\right.
\end{equation*}
where $U(x)=u(x,1)$, $B(x)=b(x,1)$, $\Theta(x)=\theta(x,1)$ and $P(x)=p(x,1)$. Consequently, a time-independent profile $(U,B,\Theta)$ can be obtained, which solves the following Leray system for the MHD-Boussinesq equations \eqref{T1.1}:
\begin{equation}\label{ES1}
		\left\{
	\begin{aligned}
		&-\Delta U-U-x\cdot\nabla U+(U\cdot\nabla)U-(B\cdot\nabla)B+\nabla P=\Theta\nabla(|\cdot|^{-1}),\\&
		-\Delta B-B-x\cdot\nabla B+(U\cdot\nabla)B-(B\cdot\nabla)U=0,\\&
		-\Delta\Theta-\Theta-x\cdot\nabla\Theta+\nabla\cdot (\Theta U)=0,\\&
		{\rm div}\,U={\rm div}\,B=0.
	\end{aligned}
	\right.
\end{equation}
Setting
\begin{align*}\label{D1-1}
	U_{0}=e^{\Delta/2}u_{0},\ \  B_{0}=e^{\Delta/2}b_{0},\ \  \mbox{and}\ \  \Theta_{0}=e^{\Delta/2}\theta_{0}.
\end{align*}
We examine the difference$(V,E,\Psi)=(U-U_{0},B-B_{0},\Theta-\Theta_{0})$, which satisfies the perturbed Leray system (see \eqref{ES2}). Hence, by applying the scaling property, the problem \eqref{E1.1}-\eqref{E1.1-initial} is reduced to a fixed-point problem:
$$
(V,E,\Psi)=\lambda S(V,E,\Psi): \mathcal{B}\rightarrow \mathcal{B}, 
$$
where $\mathcal{B}$ is a suitable Banach space and $\lambda\in[0,1]$. Our main goal is to prove that $S$ has a fixed point, we need to verify: 

{\rm (i)} $S: \mathcal{B}\rightarrow\mathcal{B}$ is a continuous and compact mapping;

{\rm (ii)} there exists a constant $C$ such that
$$
\|x\|_{\mathcal{B}} \leq C,
$$
for all $x\in\mathcal{B}$ and $\lambda\in[0,1]$ satisfying $x=\lambda Sx$.

The main difficulty in this paper is to verify the second condition: $\textit{a prior estimate}$. Motivated by \cite{KT}, we employ a blow-up argument to establish the \textit{a priori estimate} in Sobolev spaces $H^{1}$. The presence of singular gravitational force introduces difficulties, which we overcome by utilizing an appropriate cut-off function. The Leray-Schauder theorem then ensures the existence of self-similar solutions to the perturbed Leray system. Finally, applying the scaling property and embedding theorem, we derive a global forward self-similar solution to system \eqref{E1.1}.

The remainder of this paper is structured as follows. In Section 2, we introduce some basic function spaces, notations and useful lemmas. Section 3 is the core of the paper, in which we establish the existence of forward self-similar solutions to perturbed Leray system by using the blow-up argument. Section 4 is devoted to prove Theorem \ref{T1.1}. Finally, Appendix A reviews some fundamental properties of Lorentz spaces.

\section{Preliminaries}
\subsection{Functional spaces and notations}
Let $\Omega$ be an open bounded Lipschitz domain or $\R^{3}$. For any $p\in[1,\infty)$, let $L^{p}(\Omega)$(or $L^{\infty}(\Omega)$) denote the space of real functions defined on $\Omega$ with the $p$-th power absolutely integrable (or essentially bounded real functions) for the Lebesgue measure. For convenience, we write $\|\cdot\|_{L^{p}(\mathbb{R}^{3})}=\|\cdot\|_{p}$  when $\Omega=\mathbb{R}^{3}$. The Sobolev space $W^{1,2}(\Omega)$  consists of functions in $L^{2}(\Omega)$ with weak derivatives also in $L^{2}(\Omega)$. Let $C_{0}^{\infty}(\Omega)$ denote the space of smooth functions with compact support in $\Omega$. Then, we can define:
\begin{align*}
W_{0}^{1,2}(\Omega)=\overline{C_{0}^{\infty}(\Omega)}^{W^{1,2}(\Omega)}:=\mbox{the closure of} \  C_{0}^{\infty}(\Omega) \ \mbox{in}\  W^{1,2}(\Omega).
\end{align*}
As usual, we write $W^{1,2}(\Omega)=H^{1}(\Omega)$ and $W_{0}^{1,2}(\Omega)=H^{1}_{0}(\Omega)$. It is well know that when $\Omega=\R^{3}$, $H_{0}^{1}(\R^{3})=H^{1}(\R^{3})$. In this paper, bold symbols are used to denote function spaces of three-dimensional vector-valued functions. We define several commonly used spaces in the study of three-dimensional incompressible fluids:
\begin{align*}
	\mathbf{C}&_{0,\sigma}^{\infty}(\Omega):=\{\varphi\in \mathbf{C}_{0}^{\infty}(\Omega):{\rm div}\,\varphi=0\};\\
	\mathbf{L}_{\sigma}^{p}(\Omega)&=\overline{\mathbf{C}_{0,\sigma}^{\infty}(\Omega)}^{L^{p}(\Omega)},\ \ \mathbf{H}_{0,\sigma}^{1}(\Omega)=\overline{\mathbf{C}_{0,\sigma}^{\infty}(\Omega)}^{H^{1}(\Omega)}.
\end{align*}
In particular, the space $\mathbf{L}_{\sigma}^{\infty}(\Omega)$ is defined by:
\begin{equation*}
	\mathbf{L}_{\sigma}^{\infty}(\Omega)=\Big\{f\in \mathbf{L}^{\infty}(\Omega)\Big|\int_{\Omega}f\cdot\nabla\varphi dx=0,\ \mbox{for all}\ \varphi\in D^{1,1}(\Omega)\Big\},
\end{equation*}
where $D^{1,q}(\Omega)$ is the homogeneous Sobolev space of the from
$$
D^{1,q}(\Omega)=\{\varphi\in L_{loc}^{1}(\Omega)|\nabla\varphi\in \mathbf{L}^{q}(\Omega)\}.
$$
Next, we introduce a trilinear form as follow:
\begin{equation}\label{tri}
f(u,v,w)=\sum_{i,j=1}^{3}\int_{\Omega}u_{i}\partial_{i}v_{j}w_{j}dx=\int_{\Omega}u\cdot\nabla v\cdot wdx.
\end{equation}
By H\"{o}lder inequality, we obtain
$$
|f(u,v,w)|\leq \|u\|_{L^{p}(\Omega)}\|\nabla v\|_{L^{2}(\Omega)}\|u\|_{L^{q}(\Omega)}
$$
where $\frac{1}{p}+\frac{1}{q}=\frac{1}{2}$. Hence, $f$ is a trilinear continuous from on $\mathbf{L}^{p}(\Omega)\times \mathbf{H}^{1}(\Omega)\times \mathbf{L}^{q}(\Omega)$. Using straightforward calculations and embedding theorems, we provide the following results without detailed proof.
\begin{prop}\label{p11}
	Let $f$ be the trilinear form defined in \eqref{tri}, then
	
	{\rm (i)} for all $u\in \mathbf{L}_{\sigma}^{p}(\Omega)$, $p\in[3,\infty]$, and $v\in \mathbf{H}_{0,\sigma}^{1}(\Omega)$, we have
\begin{equation*}
	f(u,v,v)=0.
\end{equation*}
	{\rm (ii)} for all $u\in \mathbf{L}_{\sigma}^{p}(\Omega)$, $p\in[3,\infty]$, and $v,w\in \mathbf{H}_{0,\sigma}^{1}(\Omega)$, we have
\begin{equation*}
	f(u,v,w)=-f(u,w,v).
\end{equation*}
\end{prop}
Moreover, for convenience, we use calligraphic symbols to denote function spaces for 7-dimensional vector-valued functions. For example
\begin{align*}
	\begin{split}
		\mathcal{L}^{p}(\Omega)=\mathbf{L}^{p}(\Omega)\times\mathbf{L}^{p}(\Omega)\times L^{p}(\Omega),\\ \mathcal{C}^{\infty}_{c,\sigma}(\Omega)=\mathbf{C}^{\infty}_{c,\sigma}(\Omega)\times \mathbf{C}^{\infty}_{c,\sigma}(\Omega)\times C^{\infty}_{c}(\Omega).
	\end{split}
\end{align*}
In particular, we define the following Hilbert space
$$
\mathcal{H}(\Omega)=\mathbf{H}_{0,\sigma}^{1}(\Omega)\times \mathbf{H}_{0,\sigma}^{1}(\Omega)\times H^{1}_{0}(\Omega).
$$
Finally, we illustrate the space $BC_{w}([0,\infty),\mathcal{L}^{3,\infty}(\R^{3}))$, it is bounded and weak-$\ast$ continuous in $\mathcal{L}^{3,\infty}(\R^{3})$ with respect to time variable $t$.

\noindent\textbf{Notation:} 
In this paper, $B_{r}$ denotes an open ball of radius $r$ centered at origin in $\R^{3}$. Let $X$ be a Banach space, $X'$ represents the dual space of $X$, and the symbol $\langle\cdot,\cdot\rangle$ denotes the pairing of $X$ and $X'$. 

We denoted by $\M^{3}$ the space of all real $3\times3$ matrices. Then, If $A=(A_{ij}),B=(B_{ij})\in\M^{3}$, we write
$$
A:B=A_{ij}B_{ij}\ \ \mbox{and} \ \ |A|=(A:A)^{\frac{1}{2}},
$$
 where we use the Einstein summation convention. And, for vectors $u,v\in\R^{3}$, we denote $u\otimes v=(u_{i}v_{j})\in\M^{3}$.

\subsection{Perturbed Leray system}
 This subsection aims to introduce the perturbed Leray system of \eqref{E1.1} and define its weak solutions. We begin by recalling the definition of the heat kernel, which is the fundamental solution to the heat equation:
 $$
 G_{t}(x)=\frac{1}{(4\pi t)^{\frac{3}{2}}}e^{-\frac{|x|^{2}}{4t}},\ \ \ (x,t)\in\R^{3}\times(0,+\infty).
 $$
 Using the solution to the Cauchy problem of heat equation given as convolutions of the heat kernel with $u_{0}$, $b_{0}$ and $\theta_{0}$, we have
 \begin{equation}\label{LI}
 	\begin{cases}
 		u_{L}(x,t)=e^{t\Delta}u_{0}=G_{t}\ast u_{0},\\
 		b_{L}(x,t)=e^{t\Delta}b_{0}=G_{t}\ast b_{0},\\
 		\theta_{L}(x,t)=e^{t\Delta}\theta_{0}=G_{t}\ast \theta_{0}.
 	\end{cases}
 \end{equation} 
 If we set 
 \begin{align}\label{D1}
 	U_{0}=e^{\Delta/2}u_{0},\ \  B_{0}=e^{\Delta/2}b_{0},\ \  \mbox{and}\ \  \Theta_{0}=e^{\Delta/2}\theta_{0}.
 \end{align}
 By our assumption \eqref{initial scale} and $\mbox{div}\,u_{0}=\mbox{div}\,b_{0}=0$, we can also get the self-similar profiles $(U_{0},B_{0},\Theta_{0})$, which solve the following linear system
 \begin{equation*}\label{ES21}
 	\left\{
 		\begin{aligned}
 		&\Delta U_{0}+U_{0}+x\cdot\nabla U_{0}=0,\\&
 		\Delta B_{0}+B_{0}+x\cdot\nabla B_{0}=0,\\&
 		\Delta\Theta_{0}+\Theta_{0}+x\cdot\nabla\Theta_{0}=0,\\&
 		{\rm div}\,U_{0}={\rm div}\,B_{0}=0.
 	\end{aligned}
 		\right.
 \end{equation*}
 In order to studying system \eqref{ES1}, we consider the difference
 $$
 V=U-U_{0},\ \ E=B-B_{0},\ \  \mbox{and}\ \ \Psi=\Theta-\Theta_{0},
 $$
 which satisfy the following elliptic system, i.e., the perturbed Leray system of MHD-Boussinesq equations \eqref{E1.1}
 \begin{align}\label{ES2}
 	\left\{
 	\begin{aligned}
 		&-\Delta V-V-x\cdot\nabla V+(V+U_{0})\cdot\nabla(V+U_{0})-(E+B_{0})\cdot\nabla(E+B_{0})+\nabla P\\&\qquad\qquad\qquad\qquad\qquad\qquad\qquad\qquad\qquad\qquad\qquad\qquad=(\Psi+\Theta_{0})\nabla(|\cdot|^{-1}),\\&
 		-\Delta E-E-x\cdot\nabla E+(V+U_{0})\cdot\nabla (E+B_{0})-(E+B_{0})\cdot\nabla(V+U_{0})=0,\\&
 		-\Delta\Psi-\Psi-x\cdot\nabla\Psi+\nabla\cdot\big((\Psi+\Theta_{0})(V+U_{0})\big)=0,\\&
 		{\rm div}\,U={\rm div}\,B=0.
 	\end{aligned}
 	\right.
 \end{align}
 
	\begin{defn}(Weak solution of the perturbed Leray system)\label{D-0}
	The triple $(V,E,\Psi)$ is called a weak solution of \eqref{ES2} in $\R^{3}$, if
	
	{\rm (i)}  $(V,E,\Psi)\in \mathcal{H}_{loc}(\R^{3})$, $\nabla\cdot V=\nabla\cdot E=0$;
	
	{\rm (ii)}
\begin{align*}
	\left\{
	\begin{aligned}
			\int_{\R^{3}}\nabla V:\nabla\varphi dx&+\int_{\R^{3}}(V+U_{0})\cdot\nabla(V+U_{0})\varphi dx\\&=\int_{\R^{3}}\Big(V+x\cdot \nabla V+(E+B_{0})\cdot\nabla(E+B_{0})+(\Psi+\Theta_{0})\nabla(|\cdot|^{-1})\Big)\varphi dx,\\
			\int_{\R^{3}}\nabla E:\nabla\phi dx&+\int_{\R^{3}}(V+U_{0})\cdot\nabla(E+B_{0})\phi dx\\&=\int_{\R^{3}}\Big(E+x\cdot \nabla E+(E+B_{0})\cdot\nabla(V+U_{0})\Big)\phi dx,\\
			\int_{\R^{3}}\nabla \Psi:\nabla\zeta dx&=\int_{\R^{3}}\Big(\Psi+x\cdot\nabla\Psi\Big)\zeta dx+\int_{\R^{3}}\Big((\Psi+\Theta_{0})(V+U_{0})\Big)\nabla\zeta dx,
		\end{aligned}
	\right.
\end{align*}
	for all $\varphi,\phi\in \mathbf{C}_{0,\sigma}^{\infty}(\R^{3})$ and  $\zeta\in C_{0}^{\infty}(\R^{3})$.
\end{defn}
\begin{rem}
	Since the pressure term vanishes in Definition \ref{D-0}, we omit its discussion. In fact, by applying the divergence operator to the first equation in system \eqref{E1.1}, a pressure term can be determined by a suitable choice.
	
\end{rem}
The following lemma investigates some properties regarding self-similar solutions to the heat equation.
\begin{prop}\label{p1}
	Let  $(u_{0},b_{0},\theta_{0})\in \mathcal{L}_{loc}^{\infty}(\R^{3}\backslash\{0\})$ be (-1)-homogeneous. Then we have 
	$$
	(u_{I}, b_{I},\theta_{I})\in BC_{w}([0,+\infty),\mathcal{L}^{3,\infty}(\R^{3})).
	$$
Furthermore, for all $x\in\R^{3}$, there exist a constant $C>0$ independent on $x$, such that
\begin{align}\label{decay estimate}
	\begin{split}
		|U_{0}|+|B_{0}|+|\Theta_{0}|\leq C(1+|x|)^{-1},\\
		|\nabla U_{0}|+|\nabla B_{0}|+|\nabla\Theta_{0}|\leq C(1+|x|)^{-1},
	\end{split}
\end{align}
where $u_{I},b_{I},\theta_{I}$ and $U_{0},B_{0},\Theta_{0}$ are defined in \eqref{LI} and \eqref{D1}, respectively.
\end{prop}
\begin{proof}
Our assumption of $(u_{0},b_{0},\theta_{0})$ means that
$$
|(u_{0},b_{0},\theta_{0})|\leq \frac{C}{|x|},
$$
where the constant $C>0$ independent on $x$.
Thus, the conclusion is proved directly by Lemma \ref{Ap1}.
\end{proof}
Next, we present a critical lemma, which is important for proving $\textit{a prior estimate}$ in Section 3.
\begin{lem}\label{L0}
Let $\Omega$ be a connected domain in $\R^{3}$ with Lipschitz boundary, the function $v,b\in\mathbf{H}^{1}_{0,\sigma}(\Omega)$ and $p\in D^{1,\frac{3}{2}}(\Omega)\cap L^{3}(\Omega)$ satisfy the following system
\begin{equation}\label{L1}
\begin{cases}
(v\cdot\nabla) v-(b\cdot\nabla) b+\nabla p=0\ \ \ \ &\mbox{in}\ \ \Omega,\\
(v\cdot\nabla) b-(b\cdot\nabla) v=0\ \ \ \ &\mbox{in}\ \ \Omega,\\
{\rm div}\,v={\rm div}\,b=0\ \ \ \ \ &\mbox{in}\ \ \Omega,\\
v=b=0 \ \ \ \ &\mbox{on}\ \ \partial\Omega.
\end{cases}
\end{equation}
Then,
$$
\exists\ \  c\in\R \ \mbox{such that}\ p(x)=c\ \mbox{for $\mathfrak{H}^{2}$-almost all} \ x\in\partial\Omega,
$$
where $\mathfrak{H}^{m}$ denoted by the $m$-dimensional Hausdorff measure.
\end{lem}
\begin{rem}
Note that system \eqref{L1} has a similar structure to the stationary Euler equation. The proof of this lemma can be adapted from the results for stationary Euler equations \cite{Amick,ABLS}, with minor modifications.
	
\end{rem} 

At the end of this section, we present the Leray-Schauder theorem in a special case. Since we do not use the method of evolution equations in \cite{JS}, the nonlinear version of the Leray-Schauder theorem is not required.

\begin{thm}(\cite{GT0})\label{T2.1}
	Let $S$ be a compact mapping of a Banach space $X$ into itself, and suppose there exists a constant $M$ such that
	$$
	\|x\|_{X}<M
	$$
	for all $x\in X$ and $\lambda\in[0,1]$ satisfying $x=\lambda Sx$. Then $S$ has a fixed point.
\end{thm}
\section{Existence of solutions to the perturbed elliptic system}
\subsection{Existence of solutions to the perturbed Leray system in bounded domain}
In this subsection, we establish the existence of the solution to system \eqref{ES2} in bounded domain by the blow-up argument and fixed point principle, which is the core of this paper. 

First, for any $\lambda\in[0,1]$, we consider the following elliptic system in a bounded domain $\Omega\subset\R^{3}$ with a smooth boundary,
\begin{equation}\label{ESp2}
		\left\{
	\begin{aligned}
		&-\Delta V+\nabla P=\lambda\Big(V+x\cdot\nabla V-F_{01}-F_{1}+F_{02}+F_{2} +(\Psi+\Theta_{0})\rho\nabla(|\cdot|^{-1})\Big),\\&
		-\Delta E=\lambda\Big(E+x\cdot\nabla E-F_{03}-F_{3}+F_{04}+F_{4}\Big),\\&
		-\Delta\Psi=\lambda\Big(\Psi+x\cdot\nabla\Psi -\nabla\cdot\big((\Psi+\Theta_{0})(V+U_{0})\big)\Big),\\&
		{\rm div}\,V={\rm div}\,E=0,
	\end{aligned}
	\right.
\end{equation}
coupled with the boundary condition
\begin{equation}\label{DC1}
		V=E=\Psi=0 \ \ \mbox{on} \ \partial\Omega.
	\end{equation}
In system \eqref{ESp2}, for simplicity, we set
\begin{align}\label{Sim}
	\begin{split}
		F_{01}:=U_{0}\cdot\nabla U_{0},\ F_{02}:=B_{0}\cdot\nabla B_{0},\ F_{03}:=U_{0}\cdot\nabla B_{0},\ F_{04}:=B_{0}\cdot\nabla U_{0},\\
		\ F_{1}=(V+U_{0})\cdot\nabla V+V\cdot\nabla U_{0}, \ \ F_{2}=(E+B_{0})\cdot\nabla E+E\cdot\nabla B_{0},\\
		\ F_{3}=(V+U_{0})\cdot\nabla E+V\cdot\nabla B_{0}, \ \ F_{4}=(E+B_{0})\cdot\nabla V+E\cdot\nabla U_{0}.
	\end{split}
\end{align}
 In the first equation of system \eqref{ESp2}, $\rho$ is a smooth bounded cut-off function, which will be appropriately determined later, and we will use it to cut of the singularity at zero of the potential $|\cdot|^{-1}$.
Next, we define the following linear map $L_{\rho}$,
\begin{equation*}
\begin{aligned}
L_{\rho}(V,E,\Psi)=\Big(V+&x\cdot\nabla V-U_{0}\cdot\nabla V-V\cdot\nabla U_{0}+B_{0}\cdot\nabla E-E\cdot\nabla B_{0}+\Psi\rho\nabla(|\cdot|^{-1}),\\&\qquad\quad E+x\cdot\nabla E-U_{0}\cdot\nabla E+E\cdot\nabla U_{0}-V\cdot\nabla B_{0}+B_{0}\cdot\nabla V,\\&\qquad\qquad\qquad\qquad\qquad\qquad
\Psi+x\cdot\nabla \Psi-\nabla\cdot(\Psi U_{0}+\Theta_{0}V)\Big)
\end{aligned}
\end{equation*}
and the nonlinear map $N_{\rho}$
\begin{equation*}
\begin{aligned}
N_{\rho}(V,E,\Psi)=\Big(-U_{0}\cdot&\nabla U_{0}-V\cdot\nabla V+E\cdot\nabla E+B_{0}\cdot\nabla B_{0}+\Theta_{0}\rho\nabla(|\cdot|^{-1}),\\&\qquad\quad
-U_{0}\cdot\nabla B_{0}-V\cdot\nabla E+E\cdot\nabla V+B_{0}\cdot\nabla U_{0},\\&\qquad\qquad\qquad\qquad\qquad\qquad\quad\nabla\cdot(\Psi V+\Theta_{0}U_{0})\Big).
\end{aligned}
\end{equation*}
Then the system \eqref{ESp2} can be rewritten as
\begin{equation*}\label{re1}
(-\Delta V+\nabla P,-\Delta E,-\Delta\Psi)=\lambda A_{\rho}(V,E,\Psi),
\end{equation*}
where
\begin{equation}\label{nlm}
A_{\rho}(V,E,\Psi)=L_{\rho}(V,E,\Psi)+N_{\rho}(V,E,\Psi).
\end{equation}
Due to the Poincar\'{e} inequality, we can define the inner product of the $\mathcal{H}(\Omega)$ as follows
$$
\big\langle(V,E,\Psi),(V^{'},E^{'},\Psi^{'})\big\rangle=\int_{\Omega}\nabla V:\nabla V^{'}dx+\int_{\Omega}\nabla E:\nabla E^{'}dx+\int_{\Omega}\nabla\Psi:\nabla \Psi^{'}dx.
$$
Then the weak formula of equation can be further rewritten as
$$
\big\langle(V,E,\Psi),\Upsilon\big\rangle=\lambda\int_{\Omega}A_{\rho}\Upsilon dx, \ \ \ \forall\ \ \Upsilon\in\mathcal{C}^{\infty}_{c,\sigma}(\Omega).
$$
Denote $\mathcal{H}^{'}(\Omega)$ be the dual space of $\mathcal{H}(\Omega)$, by the Riesz representation theorem, for any $f\in\mathcal{H}^{'}(\Omega)$ there exists a isomorphism mapping $\mathbb{T}:\mathcal{H}^{'}(\Omega)\rightarrow\mathcal{H}(\Omega)$ such that
$$
\langle\mathbb{T}(f),\xi\rangle=\int_{\Omega}f\cdot\xi dx, \ \ \ \forall\ \ \xi\in\mathcal{H}(\Omega)
$$
and moreover,
$$
\|\mathbb{T}(f)\|_{\mathcal{H}(\Omega)}\leq \|f\|_{\mathcal{H}^{'}(\Omega)}.
$$
Then system \eqref{ESp2} can eventually be rewritten as
\begin{equation}\label{re2}
(V,E,\Psi)=\lambda(\mathbb{T}\circ A_{\rho})(V,E,\Psi)\triangleq \lambda S(V,E,\Psi).
\end{equation}

Now, we begin to prove the existence of a solution for system \eqref{ESp2} when $\lambda=1$. We start by establishing a priori estimates for system \eqref{ESp2} in a bounded domain with a smooth boundary.
\begin{lem}(a priori estimate)\label{L2}
Let $\Omega$ be a bounded domain in $\R^{3}$ with a smooth boundary. Assume that $(U_{0},B_{0},\Theta_{0})$ satisfy the estimates \eqref{decay estimate} and $\rho\in C^{\infty}(\R^{3})\cap L^{\infty}(\R^{3})$. Let $\lambda\in[0,1]$ and $(V,E,\Psi)\in\mathcal{H}(\Omega)$ be a solution of problem \eqref{ESp2}-\eqref{DC1}. Then we have
\begin{equation*}\label{LE1}
\int_{\Omega}(|V|^{2}+|E|^{2}+|\Psi|^{2}+|\nabla V|^{2}+|\nabla E|^{2}+|\nabla\Psi|^{2})dx\leq C_{\ast}
\end{equation*}
where the constant $C_{\ast}=C_{\ast}(\Omega,\rho,U_{0},B_{0},\Theta_{0})$, independent on $\lambda$.
\end{lem}
\begin{proof}
Let us first give a brief analysis of this problem. In fact, by Poincar\'{e} inequality, it suffices to prove
$$
\int_{\Omega}(|\nabla V|^{2}+|\nabla E|^{2}+|\nabla\Psi|^{2})dx\leq C_{\ast}.
$$
Furthermore, we claim that we only need to prove
\begin{equation}\label{LE5}
	\int_{\Omega}(|\nabla V|^{2}+|\nabla E|^{2})\leq C_{\ast}.
\end{equation}
Multiplying the third equation of $\eqref{ESp2}$ by $\Psi$ and integrating by parts in $\Omega$, we notice that 
$$
\int_{\Omega}\nabla\cdot(\Psi(V+U_{0}))\Psi dx=0,
$$
since $\nabla\cdot(V+U_{0})=0$. Thus, we obtain
\begin{equation}\label{LE2}
	\frac{\lambda}{2}\int_{\Omega}|\Psi|^{2}dx+\int_{\Omega}|\nabla\Psi|^{2}dx=-\lambda\int_{\Omega}\nabla\cdot[\Theta_{0}(V+U_{0})]\Psi dx.
\end{equation}
Using the decay estimates \eqref{decay estimate}, we can easily get that $|\Theta_{0}U_{0}|\in L^{2}(\R^{3})$. Thus, straightforward calculations show that
\begin{equation}\label{LE3}
\begin{aligned}
\lambda\Big|\int_{\Omega}\nabla\cdot\big(\Theta_{0}(V&+U_{0})\big)\Psi dx\Big|=\lambda\Big|\int_{\Omega}[\Theta_{0}(V+U_{0})]\nabla\Psi dx\Big|\\&
\leq \frac{1}{2}\int_{\Omega}|\nabla\Psi|^{2}dx+\lambda\Big(\|\Theta_{0}\|^{2}_{\infty}\int_{\Omega}|V|^{2}dx+\int_{\Omega}|\Theta_{0}U_{0}|^{2}dx\Big).
\end{aligned}
\end{equation}
By \eqref{LE2}-\eqref{LE3}, we get
\begin{equation}\label{LE4}
\frac{\lambda}{2}\int_{\Omega}|\Psi|^{2}dx+\frac{1}{2}\int_{\Omega}|\nabla\Psi|^{2}dx\leq\lambda\Big(\|\Theta_{0}\|^{2}_{\infty}\int_{\Omega}|V|^{2}dx+\int_{\Omega}|\Theta_{0}U_{0}|^{2}dx\Big).
\end{equation}
Therefore, due to \eqref{LE4}, it suffices to prove this lemma by proving \eqref{LE5}. 

Now, let us argue by contradiction: assume that there exist a sequence $\lambda_{k}\in[0,1]$ and a sequence $(V_{k},E_{k},\Psi_{k})\in\mathcal{H}(\Omega)$ satisfying
\begin{equation}\label{ESp3}
	\left\{
\begin{aligned}
&-\Delta V_{k}+\nabla P_{k}=\lambda\Big(V_{k}+x\cdot\nabla V_{k}-F_{01}-F_{1k}+F_{02}+F_{2k} +(\Psi_{k}+\Theta_{0})\rho\nabla(|\cdot|^{-1})\Big),\\&
-\Delta E_{k}=\lambda\Big(E_{k}+x\cdot\nabla E_{k}-F_{03}-F_{3k}+F_{04}+F_{4k}\Big),\\&
-\Delta\Psi_{k}=\lambda\Big(\Psi_{k}+x\cdot\nabla\Psi_{k} -\nabla\cdot\big((\Psi_{k}+\Theta_{0})(V_{k}+U_{0})\big)\Big),\\&
{\rm div}\,V_{k}={\rm div}\,E_{k}=0,
\end{aligned}
\right.
\end{equation}
where the definition of $F_{01}$, $F_{02}$, $F_{03}$, $F_{04}$ is given in \eqref{Sim} and
\begin{align}\label{Sim2}
\begin{split}
F_{1k}=(V_{k}+U_{0})\cdot\nabla V_{k}+V_{k}\cdot\nabla U_{0}, \ \ F_{2k}=(E_{k}+B_{0})\cdot\nabla E_{k}+E_{k}\cdot\nabla B_{0},\\
F_{3k}=(V_{k}+U_{0})\cdot\nabla E_{k}+V_{k}\cdot\nabla B_{0}, \ \ F_{4k}=(E_{k}+B_{0})\cdot\nabla V_{k}+E_{k}\cdot\nabla U_{0}.
\end{split}
\end{align}
Moreover,
$$
L_{k}^{2}:=\int_{\Omega}(|\nabla V_{k}|^{2}+|\nabla E_{k}|^{2})dx\rightarrow+\infty.
$$
For technical reasons, we also set
$$
J_{k}^{2}:=\int_{\Omega}|\nabla\Psi_{k}|^{2}dx.
$$
Then, we normalized $(V_{k},E_{k})$ and $\Psi_{k}$ by defining
$$
\hat{V}_{k}=\frac{V_{k}}{L_{k}},\ \ \hat{E}_{k}=\frac{E_{k}}{L_{k}},\ \ \mbox{and}\ \ \hat{\Psi}_{k}=\frac{\Psi_{k}}{J_{k}}.
$$
Since $(\hat{V}_{k},\hat{E}_{k},\hat{\Psi}_{k})$ is bounded in $\mathcal{H}(\Omega)$, we can extract a subsequence still denoted by $(\hat{V}_{k},\hat{E}_{k},\hat{\Psi}_{k})$ such that
\begin{equation}\label{wc}
(\hat{V}_{k},\hat{E}_{k},\hat{\Psi}_{k})\rightharpoonup(\bar{V},\bar{E},\bar{\Psi}),\ \ \mbox{in}\ \ \mathcal{H}(\Omega),
\end{equation}
by the compact embedding theorem, which means (after extracting a suitable subsequence)
\begin{equation}\label{sc}
(\hat{V}_{k},\hat{E}_{k},\hat{\Psi}_{k})\rightarrow(\bar{V},\bar{E},\bar{\Psi}),\ \ \mbox{in}\ \ \mathcal{L}^{p}(\Omega),\ \ \mbox{for} \ \ p\in[1,6).
\end{equation}
We can further also assume that $\lambda_{k}\rightarrow\lambda_{0}\in[0,1]$. The situation is divided into two cases depending on whether or not $L_{k}/J_{k}$ tends to infinity as $k\rightarrow+\infty$.

\textbf{Case 1:\ $\overline{\lim}_{k\rightarrow+\infty}L_{k}/J_{k}<+\infty$.} By taking a subsequence, we can assume that there exists a constant $\alpha\geq0$, such that $L_{k}/J_{k}\rightarrow\alpha$. In fact, by the estimate \eqref{LE4}, we easily deduce that $\alpha>0$. Hence we have $J_{k}\rightarrow+\infty$ because $L_{k}\rightarrow+\infty$. We then substitute $(V,\Psi)$ and $\lambda$ in equation \eqref{LE2} with $(V_{k},\Psi_{k})$ and $\lambda_{k}$, respectively, that is,
\begin{equation}\label{LE6}
\frac{\lambda_{k}}{2}\int_{\Omega}|\Psi_{k}|^{2}dx+\int_{\Omega}|\nabla\Psi_{k}|^{2}dx
=-\lambda_{k}\int_{\Omega}\nabla\cdot\Big(\Theta_{0}(V_{k}+U_{0})\Big)\Psi_{k}dx.
\end{equation}
We multiply identity \eqref{LE6} by $1/J_{k}^{2}$ and take the limit as $k\rightarrow+\infty$. Using \eqref{wc} and \eqref{sc}, we obtain
\begin{equation}\label{LEQ1}
	\begin{aligned}
	\frac{1}{J_{k}^{2}}\int_{\Omega}\nabla\cdot(\Theta_{0}V_{k})\Psi_{k}dx&=\frac{L_{k}}{J_{k}}\int_{\Omega}\nabla\cdot(\Theta_{0}\hat{V}_{k})\hat{\Psi}_{k}dx\\&
	\rightarrow\alpha\int_{\Omega}\nabla\cdot(\Theta_{0}\bar{V})\bar{\Psi}dx,
	\end{aligned}
\end{equation}
where we use the fact that $\Theta_{0}\in L^{\infty}(\Omega)$. By integration by parts and H\"{o}lder inequalities, we get
\begin{equation}\label{LEQ2}
	\begin{aligned}
\frac{1}{J_{k}^{2}}\Big|\int_{\Omega}\nabla\cdot(\Theta_{0}U_{0})\Psi_{k}dx\Big|&=\frac{1}{J_{k}^{2}}\Big|\int_{\Omega}(\Theta_{0}U_{0})\nabla\Psi_{k}dx\Big|\\&\leq \frac{C}{J_{k}}\|\Theta_{0}U_{0}\|_{L^{2}(\Omega)}\rightarrow0.
	\end{aligned}
\end{equation}
Thus, by \eqref{LE6}-\eqref{LEQ2}, we obtain the identity
\begin{equation}\label{LE7}
1+\frac{\lambda_{0}}{2}\int_{\Omega}|\bar{\Psi}|^{2}dx
=-\lambda_{0}\alpha\int_{\Omega}\nabla\cdot(\Theta_{0}\bar{V})\bar{\Psi}dx
\end{equation}
which implies that $\lambda_{0}\alpha\neq0$.
Next, we consider the weak formulation of the third equation of $\eqref{ESp3}$, namely, for all $\zeta\in C_{0}^{\infty}(\Omega)$,
\begin{equation}\label{LE8}
\int_{\Omega}\nabla\Psi_{k}\nabla\zeta dx=\lambda_{k}\Big(\int_{\Omega}(\Psi_{k}+x\cdot\nabla\Psi_{k})\zeta dx
-\int_{\Omega}\nabla\cdot\big((\Psi_{k}+\Theta_{0})(V_{k}+U_{0})\big)\zeta dx\Big).
\end{equation}
Multiplying identity \eqref{LE8} by $1/J_{k}^{2}$ and taking $k\rightarrow+\infty$. By \eqref{wc}-\eqref{sc} and the assumption of $\Theta_{0},U_{0}$, it is easy to verify that
\begin{equation*}
	\begin{aligned}
	&\frac{1}{J_{k}^{2}}\int_{\Omega}\nabla\Psi_{k}\nabla\zeta dx\rightarrow0,\\
	&\frac{\lambda_{k}}{J_{k}^{2}}\int_{\Omega}(\Psi_{k}+x\cdot\nabla\Psi_{k})\zeta dx\rightarrow0,\\
	&\frac{\lambda_{k}}{J_{k}^{2}}\int_{\Omega}\nabla\cdot\big((\Psi_{k}+\Theta_{0})U_{0}+\Theta_{0}V_{k}\big)\zeta dx\rightarrow0,
\end{aligned}
\end{equation*}
and
$$
\frac{\lambda_{k}}{J_{k}^{2}}\int_{\Omega}\nabla\cdot(\Psi_{k}V_{k})\zeta dx\rightarrow\lambda_{0}\alpha\int_{\Omega}\nabla\cdot(\bar{\Psi}\bar{V})\zeta dx.
$$
Therefore, we get
\begin{equation}\label{LEQ3}
\lambda_{0}\alpha\int_{\Omega}\nabla\cdot(\bar{\Psi}\bar{V})\zeta dx=0, \ \ \mbox{for all} \ \ \zeta\in C_{0}^{\infty}(\Omega).
\end{equation}
Since $\lambda_{0}\alpha\neq0$, the identity \eqref{LEQ3} implies $$
\nabla\cdot(\bar{\Psi}\bar{V})=\bar{V}\cdot\nabla\bar{\Psi}=0.
$$
Let us go back to \eqref{LE7}, we obtain
\begin{align*}
\begin{split}
1+\frac{\lambda_{0}}{2}\int_{\Omega}|\bar{\Psi}|^{2}dx
&=-\lambda_{0}\alpha\int_{\Omega}\nabla\cdot(\Theta_{0}\bar{V})\bar{\Psi}dx\\&=
\lambda_{0}\alpha\int_{\Omega}\Theta_{0}\bar{V}\cdot\nabla\bar{\Psi}dx=0.
\end{split}
\end{align*}
This is a contradiction, so we excludes that $\overline{\lim}_{k\rightarrow+\infty}L_{k}/J_{k}<+\infty$.

\textbf{Case 2: $\overline{\lim}_{k\rightarrow+\infty}L_{k}/J_{k}=+\infty$.} After extracting a subsequence, we can assume that $J_{k}/L_{k}\rightarrow0$. Multiplying first equation of system \eqref{ESp3} by $V_{k}$ and the second equation of system \eqref{ESp3} by $E_{k}$. By some integration by parts, we have
\begin{align}\label{LE9}
\begin{split}
&\frac{\lambda_{k}}{2}\int_{\Omega}|V_{k}|^{2}dx+\int_{\Omega}|\nabla V_{k}|^{2}dx\\&=\lambda_{k}\int_{\Omega}(B_{0}\cdot\nabla B_{0}-U_{0}\cdot\nabla U_{0})V_{k}dx+\lambda_{k}\int_{\Omega}(E_{k}\cdot\nabla B_{0}-V_{k}\cdot\nabla U_{0})V_{k}dx
\\&+\lambda_{k}\int_{\Omega}\big((E_{k}+B_{0})\cdot\nabla E_{k}-(V_{k}+U_{0})\cdot\nabla V_{k}\big)V_{k}dx\\&+\lambda_{k}\int_{\Omega}\Psi_{k}\rho\nabla(|\cdot|^{-1})V_{k}dx+\lambda_{k}\int_{\Omega}\Theta_{0}\rho\nabla(|\cdot|^{-1})V_{k}dx\\&=I_{1}+I_{2}+I_{3}+I_{4}+I_{5}
\end{split}
\end{align}
and
\begin{align}\label{LE10}
\begin{split}
&\frac{\lambda_{k}}{2}\int_{\Omega}|E_{k}|^{2}dx+\int_{\Omega}|\nabla E_{k}|^{2}dx\\&=\lambda_{k}\int_{\Omega}(B_{0}\cdot\nabla U_{0}-U_{0}\cdot\nabla B_{0})E_{k}dx+\lambda_{k}\int_{\Omega}(E_{k}\cdot\nabla U_{0}-V_{k}\cdot\nabla B_{0})E_{k}dx\\&
+\lambda_{k}\int_{\Omega}\big((E_{k}+B_{0})\cdot\nabla V_{k}-(V_{k}+U_{0})\cdot\nabla E_{k}\big)E_{k}dx\\&=I_{6}+I_{7}+I_{8}.
\end{split}
\end{align}
Let us add \eqref{LE9} and \eqref{LE10}, multiply by $1/L_{k}^{2}$ and take the limit as $k\rightarrow\infty$. With \eqref{wc}-\eqref{sc} and the assumption of $(U_{0},B_{0},\Theta_{0})$ at hand, we study the limit of each term. We can readily get
$$
\frac{1}{L_{k}^{2}}\big(I_{1}+I_{6}\big)\rightarrow0.
$$
Since $|\nabla U_{0}|,|\nabla B_{0}|\in L^{\infty}(\Omega)$, we also have
$$
\frac{1}{L_{k}^{2}}I_{2}\rightarrow\lambda_{0}\int_{\Omega}(\bar{E}\cdot\nabla B_{0}-\bar{V}\cdot\nabla U_{0})\bar{V}dx
$$
and
$$
\frac{1}{L_{k}^{2}}I_{7}\rightarrow\lambda_{0}\int_{\Omega}(\bar{E}\cdot\nabla U_{0}-\bar{V}\cdot\nabla B_{0})\bar{E}dx.
$$
Using Proposition \ref{p11} yield
$$
I_{3}+I_{8}=0.
$$
We now turn to $I_{4}$ and $I_{5}$, which involve the singular gravitational force and thus require careful handling. Since $V_{k}\in\mathbf{H}_{0}^{1}(\Omega)$ and $\Psi_{k}\in H_{0}^{1}(\Omega)$, we have
\begin{align}\label{LE11}
\begin{split}
\frac{1}{L_{k}^{2}}\Big|\int_{\Omega}\Psi_{k}\rho\nabla(|\cdot|^{-1})V_{k}dx\Big|&\leq \frac{C}{L_{k}^{2}}\bigg\|\frac{\Psi_{k}}{|\cdot|}\bigg\|_{L^{2}(\Omega)}\bigg\|\frac{V_{k}}{|\cdot|}\bigg\|_{L^{2}(\Omega)}\\&
\leq\frac{C}{L_{k}^{2}}\|\nabla\Psi_{k}\|_{L^{2}(\Omega)}\|\nabla V_{k}\|_{L^{2}(\Omega)}\\&
\leq\frac{C}{L_{k}}\|\nabla\Psi_{k}\|_{L^{2}(\Omega)}=C\frac{J_{k}}{L_{k}}\rightarrow0,
\end{split}
\end{align}
where we use the Hardy inequality. The function $\Theta_{0}$ does not belong to $H_{0}^{1}(\Omega)$, but using Lemma \ref{Lorentz}, we have
$$
\|\Theta_{0}\|_{L^{6,1}(\Omega)}\leq C\|G_{1/2}\|_{L^{\frac{6}{5},1}(\Omega)}\|\theta_{0}\|_{L^{3,\infty}(\Omega)}.
$$
 Then, using Lemma \ref{Holder}, \eqref{embedd} and the Sobolev inequality, we can drive
\begin{align}\label{LE12}
\begin{split}
\frac{1}{L_{k}^{2}}\Big|\int_{\Omega}\Theta_{0}\rho\nabla(|\cdot|^{-1})V_{k}dx\Big|&\leq \frac{C}{L_{k}^{2}}\|\Theta_{0}\|_{L^{6,1}(\Omega)}\big\|\nabla(|\cdot|^{-1})\big\|_{L^{\frac{3}{2},\infty}(\Omega)}\|V_{k}\|_{L^{6,\infty}(\Omega)}\\&
\leq\frac{C}{L_{k}^{2}}\|V_{k}\|_{L^{6}(\Omega)}\leq \frac{C}{L_{k}^{2}}\|\nabla V_{k}\|_{L^{2}(\Omega)}=\frac{C}{L_{k}}\rightarrow0.
\end{split}
\end{align}
 By the above calculations, we get the following identity
\begin{align}\label{LE13}
\begin{split}
&1+\frac{\lambda_{0}}{2}\int_{\Omega}|\bar{V}|^{2}dx+\frac{\lambda_{0}}{4}\int_{\Omega}|\bar{E}|^{2}dx\\&
=\lambda_{0}\int_{\Omega}\Big(\bar{E}\cdot\nabla B_{0}\bar{V}-\bar{V}\cdot\nabla B_{0}\bar{E}\Big)dx-\lambda_{0}\int_{\Omega}\Big(\bar{V}\cdot\nabla U_{0}\bar{V}-\bar{E}\cdot\nabla U_{0}\bar{E}\Big)dx,
\end{split}
\end{align}
which implies $\lambda_{0}\neq0$. Thus, for large enough $k$, we have $\lambda_{k}\neq0$ and we can normalize the pressure putting
$$
\hat{P_{k}}=\frac{P_{k}}{\lambda_{k}L_{k}^{2}}.
$$
Next, multiply the first and second equations of system \eqref{ESp3} by $1/(\lambda_{k}L_{k}^{2})$, yielding
\begin{align}\label{LE14}
\begin{split}
&\hat{V_{k}}\cdot\nabla\hat{V_{k}}-\hat{E_{k}}\cdot\nabla \hat{E_{k}}+\nabla\hat{P_{k}}\\&=\frac{1}{L_{k}}\bigg(\frac{\Delta \hat{V_{k}}}{\lambda_{k}}+\hat{V_{k}}+x\cdot\nabla\hat{V_{k}}+\frac{1}{L_{k}}B_{0}\cdot\nabla B_{0}-\frac{1}{L_{k}}U_{0}\cdot\nabla U_{0}         \\&-U_{0}\cdot\nabla\hat{V_{k}}-\hat{V_{k}}\cdot\nabla U_{0}+B_{0}\nabla\cdot\hat{E_{k}}+\hat{E_{k}}\cdot\nabla B_{0}+\Big(\frac{J_{k}}{L_{k}}\hat{\Psi}_{k}+\frac{1}{L_{k}}\Theta_{0}\Big)\rho\nabla(|\cdot|^{-1})\bigg)\\
\end{split}
\end{align}
and
\begin{align}\label{LE15}
\begin{split}
\hat{V_{k}}\cdot\nabla\hat{E_{k}}-\hat{E_{k}}\cdot\nabla \hat{V_{k}}=\frac{1}{L_{k}}\bigg(&\frac{\Delta \hat{E_{k}}}{\lambda_{k}}+\hat{E_{k}}+x\cdot\nabla\hat{E_{k}}+\frac{1}{L_{k}}B_{0}\cdot\nabla U_{0}-\frac{1}{L_{k}}U_{0}\cdot\nabla B_{0}  \\&         -U_{0}\cdot\nabla\hat{E_{k}}-\hat{E_{k}}\cdot\nabla B_{0}+B_{0}\cdot\nabla\hat{V_{k}}+\hat{E_{k}}\cdot\nabla U_{0}\bigg).
\end{split}
\end{align}
Now, we consider the weak formulation of the equation \eqref{LE14} and \eqref{LE15}. For convenience, we focus on equation \eqref{LE14}, as \eqref{LE15} is easier. Precisely, testing with an arbitrary $\varphi\in \mathbf{C}_{0,\sigma}^{\infty}(\Omega)$, after integrating by parts and taking $k\rightarrow\infty$. We obviously have
\begin{align*}
	\begin{split}
\frac{1}{L_{k}}\int_{\Omega}\bigg(\frac{\Delta \hat{V_{k}}}{\lambda_{k}}+\hat{V_{k}}+x\cdot\nabla\hat{V_{k}}-U_{0}\cdot\nabla\hat{V_{k}}-\hat{V_{k}}\cdot\nabla U_{0}+B_{0}\cdot\nabla\hat{E_{k}}+\hat{E_{k}}\cdot\nabla B_{0}&\bigg)\varphi dx\rightarrow0,
\end{split}
\end{align*}
and
$$
\frac{1}{L_{k}^{2}}\int_{\Omega}\big(B_{0}\cdot\nabla B_{0}-U_{0}\cdot\nabla U_{0}\big)\varphi dx\rightarrow0.
$$
Furthermore, as in \eqref{LE11}-\eqref{LE12},
$$
\frac{J_{k}}{L_{k}^{2}}\int_{\Omega}\hat{\Psi}_{k}\rho\nabla(|\cdot|^{-1})\varphi dx\rightarrow0,
$$
and
$$
\frac{1}{L_{k}^{2}}\int_{\Omega}\Theta_{0}\rho\nabla(|\cdot|^{-1})\varphi dx\rightarrow0.
$$
Combined with the fact that $\int_{\Omega}\nabla\hat{P_{k}}\cdot\varphi dx=0$, we get
\begin{equation}\label{LE16}
\int_{\Omega}(\bar{V}\cdot\nabla\bar{V}-\bar{E}\cdot\nabla\bar{E})\varphi dx=0,\ \ \ \mbox{for all}\ \ \varphi\in \mathbf{C}_{0,\sigma}^{\infty}(\Omega).
\end{equation}
Similarly, by the weak formulation of the equation \eqref{LE15},
\begin{equation}\label{LE17}
\int_{\Omega}(\bar{V}\cdot\nabla\bar{E}-\bar{E}\cdot\nabla\bar{V})\phi dx=0,\ \ \ \mbox{for all}\ \ \phi\in \mathbf{C}_{0,\sigma}^{\infty}(\Omega).
\end{equation}
Hence, \eqref{LE16} and \eqref{LE17} imply that $\bar{V},\bar{E}\in\mathbf{H}_{0,\sigma}^{1}(\Omega)$ is a weak solution of system \eqref{L1}. Then, by Rham Theorem\cite{TR} we can find a pressure $\bar{P}\in D^{1,\frac{3}{2}}(\Omega)\cap L^{3}(\Omega)$ such that $(\bar{V},\bar{E},\bar{P})$ solves
\begin{equation*}\label{LE18}
\begin{cases}
(\bar{V}\cdot\nabla) \bar{V}-(\bar{E}\cdot\nabla) \bar{E}+\nabla \bar{P}=0\ \ \ \ \ \ &\mbox{in}\ \ \Omega,\\
(\bar{V}\cdot\nabla) \bar{E}-(\bar{E}\cdot\nabla) \bar{V}=0\ \ \ \ \ \ &\mbox{in}\ \ \Omega,\\
{\rm div}\,\bar{V}={\rm div}\,\bar{E}=0\ \ \ \ \ \ &\mbox{in}\ \ \Omega,\\
\bar{V}=\bar{E}=0\ \ \ \ \ \ &\mbox{on}\ \ \partial\Omega.
\end{cases}
\end{equation*}
 Notice that $U_{0},B_{0}\in \mathbf{L}^{4}(\Omega)$, approximating $U_{0},B_{0}$ in the $\mathbf{L}^{4}-$norm by test functions and applying Proposition \ref{p11}, we obtain
\begin{equation}\label{LE19}
\int_{\Omega}\bar{E}\cdot\nabla B_{0}\bar{V}dx-\int_{\Omega}\bar{V}\cdot\nabla B_{0}\bar{E}dx=-\int_{\Omega}\bar{E}\cdot\nabla\bar{V}B_{0}dx+\int_{\Omega}\bar{V}\cdot\nabla\bar{E}B_{0}dx=0
\end{equation}
and
\begin{align}\label{LE20}
	\begin{split}
\int_{\Omega}\bar{V}\cdot\nabla U_{0}\bar{V}dx-\int_{\Omega}&\bar{E}\cdot\nabla U_{0}\bar{E}dx\\&=-\int_{\Omega}\bar{V}\cdot\nabla\bar{V}U_{0}dx+\int_{\Omega}\bar{E}\cdot\nabla\bar{E}U_{0}dx=\int_{\Omega}\nabla\bar{P}U_{0}dx.
\end{split}
\end{align}
Now, let us go back to \eqref{LE13}, by the \eqref{LE19}-\eqref{LE20}, we have
\begin{align}\label{LE21}
\begin{split}
1+\frac{\lambda_{0}}{2}\int_{\Omega}|\bar{V}|^{2}dx+\frac{\lambda_{0}}{2}\int_{\Omega}|\bar{E}|^{2}dx&=-\lambda_{0}\int_{\Omega}\nabla\bar{P}U_{0}dx\\&
=-\lambda_{0}\int_{\Omega}\nabla\cdot(\bar{P}U_{0})dx\\&=c\lambda_{0}\int_{\partial\Omega}U_{0}\cdot \overrightarrow{n}dS\\&=-c\lambda_{0}\int_{\Omega}\nabla\cdot U_{0}dx=0,
\end{split}
\end{align}
where we use the fact that the pressure $\bar{P}(x)=c$ on $\partial\Omega$ (for some constant $c\in\R$) by Lemma \ref{L0}.
This leads to a contradiction, so Case 2 does not occur, and Lemma \ref{L2} is proved.   

\end{proof}

Next, we use the Leray-Schauder theorem to prove the existence of solutions to problem \eqref{ESp2}-\eqref{DC1} in a bounded domains with smooth boundary, i.e., we need to prove the continuity and compactness of the operator $S$ in \eqref{re2}. By the definition of operator $S$, it suffices to prove that the operator $A_{\rho}$ is continuous and compact. As before, we impose additional conditions on $\rho$, assuming $\rho \in C(\R^{3})$ and satisfying
\begin{equation}\label{LE22}
\rho\in L^{\infty}(\R^{3}),\ \ 0\notin supp\rho,\ \   \rho\nabla(|\cdot|^{-1})\in L^{2}(\R^{3})\cap L^{\infty}(\R^{3}).
\end{equation}
\begin{lem}(Continuity and compactness)\label{L3}
Let $\Omega$ be a bounded domain with a smooth boundary, $\rho\in C(\R^{3})$ and satisfies \eqref{LE22}. Then
$$
A_{\rho}:\mathcal{H}(\Omega)\rightarrow\mathcal{L}^{\frac{3}{2}}(\Omega)
$$
is continuous. And
$$
A_{\rho}:\mathcal{H}(\Omega)\rightarrow\mathcal{H'}(\Omega)
$$
is compact. Where $A_{\rho}$ is defined in \eqref{nlm}.
\end{lem}
\begin{proof}
	In fact, with the help of the Sobolev embedding theorem and the assumption of $\rho$, it is straightforward to verify $A_{\rho}:\mathcal{H}(\Omega)\rightarrow\mathcal{L}^{\frac{3}{2}}(\Omega)$ is continuous.  Moreover, every function $h\in \mathcal{L}^{\frac{3}{2}}(\Omega)$ can be identified to an element of $\mathcal{H'}(\Omega)$. Thus , $A_{\rho}:\mathcal{H}(\Omega)\rightarrow\mathcal{H'}(\Omega)$ is continuous.
	
	It remains to show that $A_{\rho}:\mathcal{H}(\Omega)\rightarrow\mathcal{H'}(\Omega)$ is compact. Indeed, by the Sobolev embedding theorem, if $\|(V_{k},E_{k},\Psi_{k})\|_{\mathcal{H}(\Omega)}\leq M$, then, after extracting a subsequence (still denoted as itself), there exists $(\tilde{V},\tilde{E},\tilde{\Psi})\in \mathcal{H}(\Omega)$ such that,
$$
(v_{k},e_{k},\psi_{k})=(\tilde{V}-V_{k},\tilde{E}-E_{k},\tilde{\Psi}-\Psi_{k})\rightharpoonup0\ \ \mbox{in}\ \ \mathcal{H}(\Omega)
$$
and
$$
(v_{k},e_{k},\psi_{k})=(\tilde{V}-V_{k},\tilde{E}-E_{k},\tilde{\Psi}-\Psi_{k})\rightarrow0\ \ \mbox{in}\ \ \mathcal{L}^{3}(\Omega).
$$
For any  $\Phi=(\varphi,\phi,\zeta)\in \mathcal{H}(\Omega)$, we observe that
\begin{equation}\label{eq3-2}
	\begin{split}
	\Big|\int_{\Omega}x\cdot\nabla v_{k}\varphi dx&\Big|
	\leq\Big|\int_{\Omega}(x\otimes v_{k}):\nabla\varphi dx\Big|+3\Big|\int_{\Omega}v_{k}\varphi dx\Big|\\&\leq C\Big(\|v_{k}\|_{L^{3}(\Omega)}\|\nabla\varphi\|_{L^{2}(\Omega)}+\|v_{k}\|_{L^{3}(\Omega)}\|\varphi\|_{L^{\frac{3}{2}}(\Omega)}\Big)\\&\leq
	 C\|v_{k}\|_{L^{3}(\Omega)}\|\varphi\|_{H_{0}^{1}(\Omega)}.
	\end{split}
	\end{equation}
Thus, for the linear part $L_{\rho}$, using the condition \eqref{decay estimate}, along with certain calculations such as \eqref{eq3-2}, there exist a constant $C$ independent on $k$ and $\Phi$ such that
\begin{align*}
	\Big|\int_{\Omega}L_{\rho}&(v_{k},e_{k},\psi_{k})\cdot(\varphi,\phi,\zeta)dx\Big|\\&\leq C\big(\|v_{k}\|_{L^{3}(\Omega)}+\|e_{k}\|_{L^{3}(\Omega)}+\|\psi_{k}\|_{L^{3}(\Omega)}\big)\big(\|\varphi\|_{H_{0}^{1}(\Omega)}+\|\psi\|_{H_{0}^{1}(\Omega)}+ \|\phi\|_{H_{0}^{1}(\Omega)}\big)\\&
	\leq C\big(\|v_{k}\|_{L^{3}(\Omega)}+\|e_{k}\|_{L^{3}(\Omega)}+\|\psi_{k}\|_{L^{3}(\Omega)}\big)\|\Phi\|_{\mathcal{H}(\Omega)},
\end{align*}
where we use condition \eqref{LE22} to deal with the singularity of $\nabla(|\cdot|^{-1})$ at zero.
Then,
\begin{align}\label{Linear part}
	\begin{split}
		\Big\|L_{\rho}(\tilde{V},\tilde{E},\tilde{\Psi})-L_{\rho}(V_{k},E_{k},\Psi_{k})&\Big\|_{\mathcal{H}^{'}(\Omega)}:=\sup_{\|\Phi\|_{\mathcal{H}(\Omega)}=1}
		\Big|\int_{\Omega}L_{\rho}(v_{k},e_{k},\psi_{k})\cdot \Phi dx\Big|\\&\leq C\big(\|v_{k}\|_{L^{3}(\Omega)}+\|e_{k}\|_{L^{3}(\Omega)}+\|\psi_{k}\|_{L^{3}(\Omega)}\big)\rightarrow0.
	\end{split}
\end{align}
For the nonlinear part $N_{\rho}$, notice that
\begin{align*}
\begin{split}
&N_{\rho}(\tilde{V},\tilde{E},\tilde{\Psi})-N_{\rho}(V_{k},E_{k},\Psi_{k})\\&=\bigg(-(v_{k}+V_{k})\nabla v_{k}-v_{k}\nabla V_{k}+(e_{k}+E_{k})\nabla e_{k}+e_{k}\nabla E_{k},\\&-(v_{k}+V_{k})\nabla e_{k}-v_{k}\nabla E_{k}+(e_{k}+E_{k})\nabla v_{k}+e_{k}\nabla V_{k},\nabla\cdot\Big((\psi_{k}+\Psi_{k})v_{k}+\psi_{k}V_{k}\Big)\bigg).
\end{split}
\end{align*}
Then, for any  $\Phi=(\varphi,\phi,\zeta)\in \mathcal{H}(\Omega)$, after some integration by part, we get
\begin{align*}
\begin{split}
\Big|\int_{\Omega}N_{\rho}(\tilde{V}&,\tilde{E},\tilde{\Psi})\cdot \Phi-N_{\rho}(V_{k},E_{k},\Psi_{k})\cdot \Phi dx\Big|\\&\leq CM\big(\|v_{k}\|_{L^{3}(\Omega)}+\|e_{k}\|_{L^{3}(\Omega)}+\|\psi_{k}\|_{L^{3}(\Omega)}\big)\|\Phi\|_{\mathcal{H}(\Omega)},
\end{split}
\end{align*}
where the constant $C$ independent on $k$ and $\Phi$. It follows that
\begin{align}\label{nonlinear part}
	\begin{split}
\|N_{\rho}(\tilde{V},&\tilde{E},\tilde{\Psi})-N_{\rho}(V_{k},E_{k},\Psi_{k})\|_{\mathcal{H'}(\Omega)}
\\&=\sup_{\|\Phi\|_{\mathcal{H}(\Omega)}=1}\Big|\int_{\Omega}\Big(N_{\rho}(\tilde{V},\tilde{E},\tilde{\Psi}) -N_{\rho}(V_{k},E_{k},\Psi_{k})\Big)\cdot \Phi dx\Big|\\&\leq CM(\|v_{k}\|_{L^{3}(\Omega)}+\|e_{k}\|_{L^{3}(\Omega)}+\|\psi_{k}\|_{L^{3}(\Omega)})\rightarrow0.
\end{split}
\end{align}
Thus, by \eqref{Linear part} and \eqref{nonlinear part}, we get
\begin{align*}
\big\|A_{\rho}(V_{k},E_{k},\Psi_{k})-A_{\rho}(\tilde{V},\tilde{E},\tilde{\Psi})\big\|_{\mathcal{H'}(\Omega)}\rightarrow0,
\end{align*}
which implies that $A_{\rho}:\mathcal{H}(\Omega)\rightarrow\mathcal{H'}(\Omega)$ is compact.
\end{proof}
Next, using the notation in \eqref{Sim}, we consider the following system in a smooth bounded domain $\Omega\subseteq\R^{3}$,
\begin{equation}\label{ESp2-1}
	\left\{
	\begin{aligned}
		&-\Delta V+\nabla P=V+x\cdot\nabla V-F_{01}-F_{1}+F_{02}+F_{2} +(\Psi+\Theta_{0})\rho\nabla(|\cdot|^{-1}),\\&
		-\Delta E=E+x\cdot\nabla E-F_{03}-F_{3}+F_{04}+F_{4},\\&
		-\Delta\Psi=\Psi+x\cdot\nabla\Psi -\nabla\cdot\big((\Psi+\Theta_{0})(V+U_{0})\big),\\&
		{\rm div}\,V={\rm div}\,E=0,
	\end{aligned}
	\right.
\end{equation}
coupled with the Dirichlet boundary condition
\begin{equation}\label{DC1-1}
	V=E=\Psi=0 \ \ \mbox{on} \ \partial\Omega.
\end{equation}
By applying Lemmas \ref{L2} and \ref{L3}, we can directly derive the following results using the Leray-Schauder fixed point theorem (Theorem \ref{T2.1}).
\begin{prop}(Existence in bounded domain)\label{p2}
Let $\Omega$ be a bounded domain with a smooth boundary, $\rho\in C(\R^{3})$ and satisfies \eqref{LE22}. Assume $(U_{0},B_{0},\Theta_{0})$ satisfy \eqref{decay estimate}. Then the problem \eqref{ESp2-1}-\eqref{DC1-1} has a solution $(V,E,\Psi)\in \mathcal{H}(\Omega)$.
\end{prop}
\subsection{Existence of solutions to the perturbed Leray system in the whole space}
The goal of this subsection is to establish the existence of weak solutions for the perturbed Leray system \eqref{ES2} in the whole space. First, we prove the existence of solution to problem \eqref{ESp2-1}-\eqref{DC1-1} in any ball. In what follows, let $\rho\in C^{\infty}(\R^{3}),\ 0\leq\rho(x)\leq1$ such that $\rho(x)=0$ if $|x|\leq\frac{1}{2}$ and $\rho(x)=1$ if $|x|\geq1$. Moreover, we set $\rho_{k}(x):=\rho(kx)$. Then we have the following uniform estimate.

\begin{prop}\label{p3}
Suppose that $(U_{0},B_{0},\Theta_{0})$ satisfies \eqref{decay estimate} and $(V_{k},E_{k},\Psi_{k})\in \mathcal{H}(B_{k})$ be a solution of problem \eqref{ESp2-1}-\eqref{DC1-1} with $\Omega=B_{k}$ and $\rho=\rho_{k}$. Then we have
$$
\int_{B_{k}}(|V_{k}|^{2}+|E_{k}|^{2}+|\Psi_{k}|^{2}+|\nabla V_{k}|^{2}+|\nabla E_{k}|^{2}+|\nabla\Psi_{k}|^{2})dx\leq C_{\ast},
$$
where the constant $C_{\ast}=C_{\ast}(U_{0},B_{0},\Theta_{0})$>0, independent on $k$.
\end{prop}
\begin{proof}
	This proof is quite similar to Lemma \ref{L2}, and we sketch the proof. As in \eqref{LE4}, multiplying the third equation of $\eqref{ESp2-1}$ by $\Psi_{k}$ yields
\begin{equation}\label{LE25}
\frac{1}{2}\int_{B_{k}}|\Psi_{k}|^{2}dx+\frac{1}{2}\int_{B_{k}}|\nabla\Psi_{k}|^{2}dx\leq
\Big(\|\Theta_{0}\|^{2}_{\infty}\int_{B_{k}}|V_{k}|^{2}dx+\int_{\R^{3}}|\Theta_{0}U_{0}|^{2}dx\Big).
\end{equation}
Thus, we only need to prove that
\begin{equation}\label{LE26}
\int_{B_{k}}\Big(\frac{1}{2}|V_{k}|^{2}+\frac{1}{2}|E_{k}|^{2}+|\nabla V_{k}|^{2}+|\nabla E_{k}|^{2}\Big)dx\leq C_{\ast}.
\end{equation}
We argue by contradiction, assume that \eqref{LE26} does not hold. Then, after taking a subsequence, there exists a sequence of solutions $(V_{k},E_{k},\Psi_{k})\in\mathcal{H}(B_{k})$ to problem \eqref{ESp2-1}-\eqref{DC1-1} with $\Omega=B_{k}$ and $\rho=\rho_{k}$ such that
$$
L_{k}^{2}:=\int_{B_{k}}\Big(\frac{1}{2}|V_{k}|^{2}+\frac{1}{2}|E_{k}|^{2}+|\nabla V_{k}|^{2}+|\nabla E_{k}|^{2}\Big)dx\rightarrow+\infty.
$$
Moreover, we also set
$$
J_{k}^{2}:=\int_{B_{k}}\Big(\frac{1}{2}|\Psi_{k}|^{2}+|\nabla \Psi_{k}|^{2}\Big)dx.
$$
Consider the normalized sequence of $(V_{k},E_{k})$ and $\Psi_{k}$
\begin{equation}\label{LE27}
\hat{V}_{k}=\frac{V_{k}}{L_{k}},\ \ \hat{E}_{k}=\frac{E_{k}}{L_{k}},\ \ \mbox{and}\ \ \hat{\Psi}_{k}=\frac{\Psi_{k}}{J_{k}}.
\end{equation}
By the definition \eqref{LE27}, the sequence $(\hat{V}_{k},\hat{E}_{k},\hat{\Psi}_{k})$ is bounded in $\mathcal{H}(B_{k})$. Thus, via the classical extension theorem, there exists $(\bar{V},\bar{E},\bar{\Psi})\in\mathcal{H}(\R^{3})$, such that
\begin{equation}\label{weak con}
(\hat{V}_{k},\hat{E}_{k},\hat{\Psi}_{k})\rightharpoonup(\bar{V},\bar{E},\bar{\Psi}),\ \ \mbox{in}\ \ \mathcal{H}(\R^{3})
\end{equation}
and
\begin{equation}\label{strong con}
(\hat{V}_{k},\hat{E}_{k},\hat{\Psi}_{k})\rightarrow(\bar{V},\bar{E},\bar{\Psi}),\ \ \mbox{locally in}\ \ \mathcal{L}^{p}(\R^{3})\ \ \mbox{for all}\ \ 1\leq p<6.
\end{equation}
As in Lemma \ref{L2}, the situation is divided into two cases depending on whether or not $L_{k}/J_{k}$ tends to infinity as $k\rightarrow+\infty$.

\textbf{Case 1: $\overline{\lim}_{k\rightarrow+\infty}L_{k}/J_{k}<+\infty$.}
  We may assume that $L_{k}/J_{k}\rightarrow\alpha\in[0,\infty)$ by taking a subsequence. Since $L_{k}\rightarrow+\infty$, we must have $J_{k}\rightarrow\infty$. Thus, \eqref{LE25} implies that $\alpha>0$. Similar to \eqref{LE2}, we also have
\begin{equation}\label{LE28}
\frac{1}{2}\int_{B_{k}}|\Psi_{k}|^{2}dx+\int_{B_{k}}|\nabla\Psi_{k}|^{2}dx=-\int_{B_{k}}\nabla\cdot\big(\Theta_{0}(V_{k}+U_{0})\big)\Psi_{k}dx.
\end{equation}
Multiplying \eqref{LE28} by $1/J_{k}^{2}$ and taking $k\rightarrow\infty$, we can easily get
$$
\frac{1}{J_{k}^{2}}\Big(\frac{1}{2}\int_{B_{k}}|\Psi_{k}|^{2}+\int_{B_{k}}|\nabla\Psi_{k}|^{2}\Big)dx=1
$$
and
$$
\frac{1}{J_{k}^{2}}\int_{B_{k}}\nabla\cdot(\Theta_{0}U_{0})\Psi_{k}dx\rightarrow0.
$$
For the last term of \eqref{LE28}, we claim that
\begin{equation}\label{LE29}
\frac{1}{J_{k}^{2}}\int_{B_{k}}\nabla\cdot(\Theta_{0}V_{k})\Psi_{k}dx=\frac{1}{J_{k}^{2}}\int_{B_{k}}\nabla\Theta_{0}\cdot V_{k}\Psi_{k}dx\rightarrow\alpha\int_{\R^{3}}\nabla\Theta_{0}\cdot\bar{V}\bar{\Psi}dx.
\end{equation}
Notice that,
\begin{align*}\label{LE30}
\begin{split}
&\Big|\frac{1}{J_{k}^{2}}\int_{B_{k}}\nabla\Theta_{0}\cdot V_{k}\Psi_{k}dx-\alpha\int_{\R^{3}}\nabla\Theta_{0}\cdot\bar{V}\bar{\Psi}dx\Big|\\&\leq
\Big|\int_{B_{k}}\Big(\frac{L_{k}}{J_{k}}\nabla\Theta_{0}\cdot \hat{V}_{k}\hat{\Psi}_{k}-\alpha\nabla\Theta_{0}\cdot\bar{V}\bar{\Psi}\Big)dx\Big|+
\alpha\|\nabla\Theta_{0}\|_{L^{4}(B_{k}^{c})}\|\bar{V}\|_{L^{\frac{8}{3}}(\R^{3})}\|\bar{\Psi}\|_{L^{\frac{8}{3}}(\R^{3})}\\&=I_{1}+I_{2}.
\end{split}
\end{align*}
The assume of $\Theta_{0}$ implies $|\nabla\Theta_{0}|\in L^{4}(\R^{3})$, thus $I_{2}\rightarrow0$ as $k\rightarrow\infty$.
On the other hand,
for any bounded domain $\Omega$, we can take $k$ large enough such that $\Omega\subset B_{k}$. Then, for any $\varepsilon>0$, since the Banach extension theorem is bounded, using \eqref{strong con} we have
\begin{align*}
	\begin{split}
I_{1}&\leq \int_{\Omega}\Big|\frac{L_{k}}{J_{k}}\nabla\Theta_{0}\cdot \hat{V}_{k}\hat{\Psi}_{k}-\alpha\nabla\Theta_{0}\cdot\bar{V}\bar{\Psi}\Big|dx+\int_{\Omega^{c}}\Big|\frac{L_{k}}{J_{k}}\nabla\Theta_{0}\cdot \hat{V}_{k}\hat{\Psi}_{k}-\alpha\nabla\Theta_{0}\cdot\bar{V}\bar{\Psi}\Big|dx\\&\leq  \int_{\Omega}\Big|\Big(\frac{L_{k}}{J_{k}}-\alpha\Big)\nabla\Theta_{0}\cdot \hat{V}_{k}\hat{\Psi}_{k}\Big|dx+\int_{\Omega}\Big|\alpha\nabla\Theta_{0}\cdot\big(\hat{V}_{k}-\bar{V}\big)\hat{\Psi}_{k}\Big|dx\\&+\int_{\Omega}\Big|\alpha\nabla\Theta_{0}\cdot\bar{V}_{k}\big(\hat{\Psi}_{k}-\bar{\Psi}\big)\Big|dx+C\|\nabla\Theta_{0}\|_{L^{4}(\Omega^{c})}\\&\leq C\bigg(\Big|\frac{L_{k}}{J_{k}}-\alpha\Big|+\|\hat{V}_{k}-\bar{V}\|_{L^{\frac{8}{3}}(\Omega)}+\|\hat{\Psi}_{k}-\bar{\Psi}\|_{L^{\frac{8}{3}}(\Omega)}\bigg)+C\|\nabla\Theta_{0}\|_{L^{4}(\Omega^{c})}\\&<\varepsilon,
\end{split}
\end{align*}
where we can take $\Omega=B_{R}$, $R<k$ large enough. Hence, \eqref{LE29} was proved, and go back to \eqref{LE28}, we have
\begin{equation}\label{LE300}
1=-\alpha\int_{\R^{3}}\nabla\cdot(\Theta_{0}\bar{V})\bar{\Psi}dx.
\end{equation}
Now, for any $\zeta\in C_{0}^{\infty}(\R^{3})$, we consider the weak formulation of the third equation of system $\eqref{ESp3}$, 
\begin{equation}\label{LE31}
\int_{B_{k}}\nabla\Psi_{k}\nabla\zeta dx=\int_{B_{k}}(\Psi_{k}+x\cdot\nabla\Psi_{k})\zeta dx-\int_{B_{k}}\nabla\cdot\big((\Theta_{0}+\Psi_{k})(V_{k}+U_{0})\big)\zeta dx,
\end{equation}
where we take $k$ large enough such that $\supp\phi\subset B_{k}$. Dividing by $J_{k}^{2}$ in the equation \eqref{LE31} and letting $k\rightarrow+\infty$, as in \eqref{LE29},
$$
\frac{1}{L_{k}^{2}}\int_{B_{k}}\nabla\cdot(\Psi_{k}V_{k})\phi dx\rightarrow\alpha\int_{\R^{3}}\nabla(\bar{\Psi}\bar{V})\phi dx.
$$
The other term in \eqref{LE31} tend to 0 as $k\rightarrow+\infty$. Then it follows that
$$
\alpha\int_{\R^{3}}\nabla\cdot(\bar{\Psi}\bar{V})\phi dx=0,\ \ \ \mbox{for all}\ \phi\in C_{0}^{\infty}(\R^{3}).
$$
Using the fact that $\alpha>0$, we obtain $\nabla\cdot(\bar{\Psi}\bar{V})=\bar{V}\cdot\nabla\bar{\Psi}=0$. However, by \eqref{LE300}, this leads to the contradiction:
$$
1=-\alpha\int_{\R^{3}}\nabla\cdot(\Theta_{0}\bar{V})\bar{\Psi}dx=\alpha\int_{\R^{3}}\Theta_{0}\bar{V}\cdot\nabla\bar{\Psi}dx=0.
$$
Therefore, we excludes Case 1.

\textbf{Case 2: $\overline{\lim}_{k\rightarrow+\infty}L_{k}/J_{k}=+\infty$.}
After extracting a subsequence, we assume $J_{k}/L_{k}\rightarrow0$ as $k\rightarrow+\infty$. Multiplying first equation of system \eqref{ESp2-1} by $V_{k}$ and the second equation of system \eqref{ESp2-1} by $E_{k}$, we get
\begin{align}\label{LE32}
	\begin{split}
		&\frac{1}{2}\int_{B_{k}}|V_{k}|^{2}dx+\int_{B_{k}}|\nabla V_{k}|^{2}dx\\&=\int_{B_{k}}(B_{0}\cdot\nabla B_{0}-U_{0}\cdot\nabla U_{0})V_{k}dx+\int_{B_{k}}(E_{k}\cdot\nabla B_{0}-V_{k}\cdot\nabla U_{0})V_{k}dx
		\\&+\int_{B_{k}}\big((E_{k}+B_{0})\cdot\nabla E_{k}-(V_{k}+U_{0})\cdot\nabla V_{k}\big)V_{k}dx\\&+\int_{B_{k}}\Psi_{k}\rho_{k}\nabla(|\cdot|^{-1})V_{k}dx+\int_{B_{k}}\Theta_{0}\rho_{k}\nabla(|\cdot|^{-1})V_{k}dx
	\end{split}
\end{align}
and
\begin{align}\label{LE33}
	\begin{split}
		&\frac{1}{2}\int_{B_{k}}|E_{k}|^{2}dx+\int_{B_{k}}|\nabla E_{k}|^{2}dx\\&=\int_{B_{k}}(B_{0}\cdot\nabla U_{0}-U_{0}\cdot\nabla B_{0})E_{k}dx+\int_{B_{k}}(E_{k}\cdot\nabla U_{0}-V_{k}\cdot\nabla B_{0})E_{k}dx\\&
		+\int_{B_{k}}\big((E_{k}+B_{0})\cdot\nabla V_{k}-(V_{k}+U_{0})\cdot\nabla E_{k}\big)E_{k}dx.
	\end{split}
\end{align}
Let us add \eqref{LE32} and \eqref{LE33}, multiply by $1/L_{k}^{2}$ and taking the limit as $k\rightarrow\infty$. By \eqref{weak con}-\eqref{strong con}, following similar calculations as in Case 2 of Lemma \ref{L2}, we have
\begin{align}\label{LE34}
\begin{split}
1=\int_{\R^{3}}\bar{E}\cdot\nabla U_{0}\bar{E}dx-\int_{\R^{3}}\bar{V}\cdot\nabla U_{0}\bar{V}dx+\int_{\R^{3}}\bar{E}\cdot\nabla B_{0}\bar{V}dx-\int_{\R^{3}}\bar{V}\cdot\nabla B_{0}\bar{E}dx,
\end{split}
\end{align}
where we use the fact that $J_{k}/L_{k}\rightarrow0$ as $k\rightarrow+\infty$.
Then we multiply by $1/L_{k}^{2}$ of the first equation and second equation of \eqref{ESp2-1}, we obtain
\begin{align}\label{LE141}
	\begin{split}
		\hat{V_{k}}\cdot\nabla&\hat{V_{k}}-\hat{E_{k}}\cdot\nabla \hat{E_{k}}+\nabla\hat{P_{k}}\\&=\frac{1}{L_{k}}\bigg(\Delta \hat{V_{k}}+\hat{V_{k}}+x\cdot\nabla\hat{V_{k}}+\frac{1}{L_{k}}B_{0}\cdot\nabla B_{0}-\frac{1}{L_{k}}U_{0}\cdot\nabla U_{0}         \\&-U_{0}\nabla\hat{V_{k}}-\hat{V_{k}}\nabla U_{0}+B_{0}\nabla\hat{E_{k}}+\hat{E_{k}}\nabla B_{0}+\Big(\frac{J_{k}}{L_{k}}\hat{\Psi}_{k}+\frac{1}{L_{k}}\Theta_{0}\Big)\rho_{k}\nabla(|\cdot|^{-1})\bigg)\\
	\end{split}
\end{align}
and
\begin{align}\label{LE151}
	\begin{split}
		\hat{V_{k}}\cdot\nabla\hat{E_{k}}-\hat{E_{k}}\cdot\nabla \hat{V_{k}}=\frac{1}{L_{k}}\bigg(\Delta \hat{E_{k}}&+\hat{E_{k}}+x\cdot\nabla\hat{E_{k}}+\frac{1}{L_{k}}B_{0}\cdot\nabla U_{0}-\frac{1}{L_{k}}U_{0}\cdot\nabla B_{0}  \\&         -U_{0}\nabla\hat{E_{k}}-\hat{E_{k}}\nabla B_{0}+B_{0}\nabla\hat{V_{k}}+\hat{E_{k}}\nabla U_{0}\bigg).
	\end{split}
\end{align}
We consider the weak formulation of \eqref{LE141} and \eqref{LE151} by testing with arbitrary $\varphi\in \mathbf{C}_{0,\sigma}^{\infty}(\R^{3})$ and $\phi\in \mathbf{C}_{0,\sigma}^{\infty}(\R^{3})$, respectively. Taking the limit as $k\rightarrow+\infty$ yields
\begin{equation*}\label{LE36}
\int_{\R^{3}}(\bar{V}\cdot\nabla\bar{V}-\bar{E}\cdot\nabla\bar{E})\varphi dx=0,\ \ \ \mbox{for all}\ \ \varphi\in \mathbf{C}_{0,\sigma}^{\infty}(\R^{3})
\end{equation*}
and
\begin{equation*}\label{LE37}
\int_{\R^{3}}(\bar{V}\cdot\nabla\bar{E}-\bar{E}\cdot\nabla\bar{V})\phi dx=0,\ \ \ \mbox{for all}\ \ \phi\in \mathbf{C}_{0,\sigma}^{\infty}(\R^{3}).
\end{equation*}
Then, approximating $U_{0}$ by the test function $\varphi\in \mathbf{C}_{0,\sigma}^{\infty}(\R^{3})$ in $L^{4}$-norm and approximate $B_{0}$ by the test function $\phi\in \mathbf{C}_{0,\sigma}^{\infty}(\R^{3})$ in $L^{4}$-norm implies that
$$
-\int_{\R^{3}}\bar{V}\cdot\nabla U_{0}\bar{V}dx+\int_{\R^{3}}\bar{E}\cdot\nabla U_{0}\bar{E}dx=\int_{\R^{3}}(\bar{V}\cdot\nabla\bar{V}-\bar{E}\cdot\nabla\bar{E})U_{0}dx=0
$$
and
$$
\int_{\R^{3}}\bar{E}\cdot\nabla B_{0}\bar{V}dx-\int_{\R^{3}}\bar{V}\cdot\nabla B_{0}\bar{E}dx=\int_{\R^{3}}(\bar{V}\cdot\nabla\bar{E}-\bar{E}\cdot\nabla\bar{V})B_{0}dx=0,
$$
which contradicts \eqref{LE34}. Thus we excludes Case 2 and complete the proof of the Proposition.
\end{proof}
\begin{thm}(Existence in the whole space)\label{T1.2}
Let $(U_{0},B_{0},\Theta_{0})$ satisfies \eqref{decay estimate}. Then the perturbed Leray system \eqref{ES2} has a weak solution $(V,E,\Psi)\in \mathcal{H}(\R^{3})$.
\end{thm}
\begin{proof}
By Proposition \ref{p2}, there exists a sequence of solutions $(V_{k},E_{k},\Psi_{k})\in \mathcal{H}(B_{k})$ to problem \eqref{ESp2-1}-\eqref{DC1-1} with $\Omega=B_{k}$ and $\rho=\rho_{k}$.
Furthermore, Proposition \ref{p3} enable us to get
$$
\|(V_{k},E_{k},\Psi_{k})\|_{\mathcal{H}(B_{k})}\leq C_{\ast},
$$
where the constant $C_{\ast}=C_{\ast}(U_{0},B_{0},\Theta_{0})>0$ independent on $k$. Thus, after extract a subsequence, there exist $(V,E,\Psi)\in \mathcal{H}(\R^{3})$ such that
\begin{equation}\label{weak con1}
(V_{k},E_{k},\Psi_{k})\rightharpoonup(V,E,\Psi), \ \ \mbox{in}\ \ \mathcal{H}(\R^{3})
\end{equation}
and
\begin{equation}\label{strong con1}
(V_{k},E_{k},\Psi_{k})\rightarrow(V,E,\Psi), \ \ \mbox{locally in}\ \ \mathcal{L}^{p}(\R^{3}),\ \ \mbox{for any} \ \ p\in[1,6).
\end{equation}
It remains to verify that the limit function $(V,E,\Psi)$ is a weak solution of the system \eqref{ES2} in $\R^{3}$.
We consider the weak formulation of problem \eqref{ESp2-1}-\eqref{DC1-1} in $B_{k}$ with $\rho=\rho_{k}$, and let $k\rightarrow+\infty$. For convenience, we only consider the first equation of \eqref{ESp2-1}, the second and third equation can be proved in a similar way and are even easier. Using  \eqref{weak con1} and \eqref{strong con1}, for any $\varphi\in \mathbf{C}_{0,\sigma}^{\infty}(\R^{3})$, taking $k$ large enough such that $\tilde{\Omega}=\supp\varphi\subset B_{k}$, we get
\begin{align}\label{E33-1}
	\begin{split}
		\int_{\R^{3}}&\nabla V_{k}:\nabla\varphi dx-\int_{\R^{3}}V_{k}\varphi dx-\int_{\R^{3}}x\cdot\nabla V_{k}\varphi dx\\&
		\rightarrow\int_{\R^{3}}\nabla V:\nabla\varphi dx-\int_{\R^{3}}V\varphi dx-\int_{\R^{3}}x\cdot\nabla V\varphi dx.
	\end{split}
\end{align}
Notice that,
\begin{align*}
	\begin{split}
		\int_{\R^{3}}V_{k}&\cdot\nabla V_{k}\varphi dx-\int_{\R^{3}}V\cdot\nabla V\varphi dx\\&=\int_{\R^{3}}\big(V_{k}\otimes (V-V_{k})\big)\nabla\varphi dx+\int_{\R^{3}}\big((V-V_{k})\otimes V\big)\nabla\varphi dx\\&\leq\|V_{k}\|_{2}\|V-V_{k}\|_{L^{3}(\tilde{\Omega})}\|\nabla\varphi\|_{6}+\|V-V_{k}\|_{L^{3}(\tilde{\Omega})}\|V\|_{2}\|\nabla\varphi\|_{6}\\&\longrightarrow0,\ \mbox{as}\ \  k\rightarrow\infty,
	\end{split}
\end{align*}
where we use the H\"{o}der inequality and \eqref{strong con1}.
Thus,
\begin{equation}\label{E3.3}
	\int_{\R^{3}}V_{k}\cdot\nabla V_{k}\varphi dx\rightarrow\int_{\R^{3}}V\cdot\nabla V\varphi dx.
\end{equation}
Similarly, we also have
\begin{equation}\label{E3.4}
	\int_{\R^{3}}E_{k}\cdot\nabla E_{k}\varphi dx\rightarrow\int_{\R^{3}}E\cdot\nabla E\varphi dx.
\end{equation}
By \eqref{decay estimate}, we get $|U_{0}|,|\nabla U_{0}|\in L^{\infty}(\R^{3})$, which allows us to get
\begin{align}\label{E34}
	\begin{split}
		&\int_{\R^{3}}U_{0}\nabla V_{k}\varphi dx+\int_{\R^{3}}V_{k}\nabla U_{0}\varphi dx\\&=-\int_{\R^{3}}U_{0}\otimes V_{k}\nabla\varphi dx-\int_{\R^{3}}V_{k}\nabla U_{0}\varphi dx\\&\rightarrow-\int_{\R^{3}}U_{0}\otimes V\nabla\varphi dx-\int_{\R^{3}}V\nabla U_{0}\varphi dx\\&=\int_{\R^{3}}U_{0}\nabla V\varphi dx+\int_{\R^{3}}V\nabla U_{0}\varphi dx.
	\end{split}
\end{align}
As above, since $|B_{0}|,|\nabla B_{0}|\in L^{\infty}(\R^{3})$, we also have
\begin{align}\label{E35}
	\begin{split}
		\int_{\R^{3}}B_{0}\cdot\nabla E_{k}\varphi dx&+\int_{\R^{3}}E_{k}\cdot\nabla B_{0}\varphi dx\\&\rightarrow\int_{\R^{3}}B_{0}\cdot\nabla E\varphi dx+\int_{\R^{3}}E\cdot\nabla B_{0}\varphi dx.
	\end{split}
\end{align}
Finally, we consider the term which contains the singular gravitational force, we claim that
\begin{equation}\label{claim}
\int_{\R^{3}}\Psi_{k}\rho_{k}\nabla(|\cdot|^{-1})\varphi dx\rightarrow\int_{\R^{3}}\Psi\nabla(|\cdot|^{-1})\varphi dx.
\end{equation}
Notice that,
\begin{align*}
	\begin{split}
\Big|\int_{\R^{3}}\Psi_{k}\rho_{k}&\nabla(|\cdot|^{-1})\varphi dx-\int_{\R^{3}}\Psi\nabla(|\cdot|^{-1})\varphi dx\Big|\\&\leq \Big|\int_{\R^{3}}(\Psi_{k}-\Psi)\rho_{k}\nabla(|\cdot|^{-1})\varphi dx\Big|+ \Big|\int_{\R^{3}}\Psi(\rho_{k}-\rho)\nabla(|\cdot|^{-1})\varphi dx\Big|\\&=I_{1}+I_{2}.
\end{split}
\end{align*}
By Lemma \ref{Holder} and \eqref{strong con1}, we have
$$
I_{1}\leq \|\Psi_{k}-\Psi\|_{L^{4}(\tilde{\Omega})}\|\nabla(|\cdot|^{-1})\|_{L^{\frac{3}{2},\infty}(\R^{3})}\|\varphi\|_{L^{12,\frac{4}{3}}(\R^{3})}\rightarrow0.
$$
And since $\rho_{k}(x)\rightarrow1$ a.e. in $\R^{3}$ as $k\rightarrow+\infty$, we get
$
I_{2}\rightarrow0.
$
Thus, \eqref{claim} is proved. Similarly, by \eqref{decay estimate}, $\Theta_{0}\in L^{4}(\R^{3})$, we also have
\begin{equation}\label{claim1}
	\int_{\R^{3}}\Theta_{0}\rho_{k}\nabla(|\cdot|^{-1})\varphi dx\rightarrow\int_{\R^{3}}\Theta_{0}\nabla(|\cdot|^{-1})\varphi dx.
\end{equation}
Collecting \eqref{E33-1}-\eqref{claim1}, 
\begin{align}\label{E36-1}
	\begin{split}
		\int_{\R^{3}}&\nabla V:\nabla\varphi dx+\int_{\R^{3}}(V+U_{0})\cdot\nabla(V+U_{0})\varphi dx\\&=\int_{\R^{3}}\Big(V+x\cdot \nabla V+(E+B_{0})\cdot\nabla(E+B_{0})+(\Psi+\Theta_{0})\nabla(|\cdot|^{-1})\Big)\varphi dx.
	\end{split}
\end{align}
Meanwhile, considering the weak formulation of the second and third equation in system \eqref{ESp2-1} with $\Omega=B_{k}$ and $\rho=\rho_{k}$, taking $k\rightarrow+\infty$, we get
\begin{align}\label{E37-1}
	\begin{split}
\int_{\R^{3}}\nabla E:\nabla\phi dx&+\int_{\R^{3}}(V+U_{0})\cdot\nabla(E+B_{0})\phi dx\\&=\int_{\R^{3}}\Big(E+x\cdot \nabla E+(E+B_{0})\cdot\nabla(V+U_{0})\Big)\phi dx
	\end{split}
\end{align}
and
\begin{equation}\label{E38-1}
\int_{\R^{3}}\nabla \Psi:\nabla\zeta dx=\int_{\R^{3}}\Big(\Psi+x\cdot\nabla\Psi\Big)\zeta dx+\int_{\R^{3}}\Big((\Psi+\Theta_{0})(V+U_{0})\Big)\nabla\zeta dx
\end{equation}
	for all $\phi\in \mathbf{C}_{0,\sigma}^{\infty}(\R^{3})$ and  $\zeta\in C_{0}^{\infty}(\R^{3})$.
By \eqref{E36-1}-\eqref{E38-1}, we conclude that $(V,E,\Psi)$ is a weak solution to the elliptic system \eqref{ES2} in the whole space.
\end{proof}

\section{Proof of Theorem \ref{T1.1}}
In this section, we complete the proof of Theorem \ref{T1.1}.
\begin{proof}[Proof of Theorem \ref{T1.1}] First, we focus on the existence of forward self-similar solution. Setting
\begin{equation*}
	\left\{
\begin{aligned}
v(x,t)=\frac{1}{\sqrt{2t}}V\Big(\frac{x}{\sqrt{2t}}\Big),\\
e(x,t)=\frac{1}{\sqrt{2t}}E\Big(\frac{x}{\sqrt{2t}}\Big),\\
\psi(x,t)=\frac{1}{\sqrt{2t}}\Psi\Big(\frac{x}{\sqrt{2t}}\Big),
\end{aligned}
\right.
\end{equation*}
where $(V,E,\Psi)$ was constructed in Theorem \ref{T1.2}. Let
$(u_{L},b_{L},\theta_{L})$ be as in \eqref{LI}, if we define
\begin{equation*}
	\left\{
\begin{aligned}
u(x,t)=u_{L}(x,t)+v(x,t),\\
b(x,t)=b_{L}(x,t)+e(x,t),\\
\theta(x,t)=\theta_{L}(x,t)+\psi(x,t).
\end{aligned}
\right.
\end{equation*}
Then, by the scaling properties, one can get $(u,b,\theta)$ is a solution of the MHD-Boussinesq system \eqref{E1.1}. 
We now begin to prove that
$$
\big(u(x,t), b(x,t),\theta(x,t)\big)\in BC_{w}([0,+\infty),\mathcal{L}^{3,\infty}(\R^{3})).
$$
By Proposition \ref{p1}, we have
\begin{equation}\label{E4.1}
\big(u_{I}(x,t), b_{I}(x,t),\theta_{I}(x,t)\big)\in BC_{w}([0,+\infty),\mathcal{L}^{3,\infty}(\R^{3}))
\end{equation}
and
\begin{equation*}
	\big(u_{I}(\cdot,t), b_{I}(\cdot,t),\theta_{I}(\cdot,t)\big)\rightarrow(u_{0},b_{0},\theta_{0}),\ \ \mbox{as}\ \ t\rightarrow0^{+}
\end{equation*}
in the weak star topology of $\mathcal{L}^{3,\infty}(\R^{3})$. 	It remain to prove that
\begin{equation*}\label{E4.2}
	\big(v(x,t), e(x,t),\psi(x,t)\big)\in BC_{w}([0,+\infty),\mathcal{L}^{3,\infty}(\R^{3})).
\end{equation*}
Since $(V,E,\Psi)\in\mathcal{H}(\R^{3})$, the Sobolev embedding theorem enable us to get $(V,E,\Psi)\in\mathcal{L}^{p}(\R^{3})$ (for any $2\leq p\leq6$), in particular,
$$
\big\|\big(v(t),e(t),\psi(t)\big)\big\|_{\mathcal{L}^{3}(\R^{3})}=\big\|(V,E,\Psi)\big\|_{\mathcal{L}^{3}(\R^{3})}\leq C\big\|(V,E,\Psi)\big\|_{\mathcal{H}(\R^{3})},
$$
which implies that $\big(v(t),e(t),\psi(t)\big)\in L^{\infty}([0,\infty),\mathcal{L}^{3,\infty}(\R^{3}))$. Next, we consider the continuity of $\big(v(x,t), e(x,t),\psi(x,t)\big)$.
 In fact, we only need to consider the continuity at $t=0$.
By the interpolation theory (see Appendix A), any function $\Upsilon\in \mathcal{L}^{\frac{3}{2},1}(\R^{3})$ can be approximated by functions $\Upsilon_{\epsilon}\in \mathcal{L}^{1}\cap \mathcal{L}^{2}(\R^{3})$.
Notice that, for all $2\leq p\leq6$,
\begin{equation}\label{E4.3}
	\begin{aligned}
		\Big\|\big(v(t),e(t),\psi(t)\big)\Big\|_{\mathcal{L}^{p}(\R^{3})}&=\bigg(\int_{\R^{3}}\Big|\frac{1}{\sqrt{2t}}(V,E,\Psi)\Big(\frac{x}{\sqrt{2t}}\Big)\Big|^{p}dx\bigg)^{\frac{1}{p}}\\&
		=(2t)^{\frac{3}{2p}-\frac{1}{2}}\Big\|(V,E,\Psi)\Big\|_{\mathcal{L}^{p}(\R^{3})}.
	\end{aligned}
\end{equation}
Hence,
$$
\int_{\R^{3}}\big(v(t),e(t),\psi(t)\big)\Upsilon_{\epsilon}dx\leq (2t)^{\frac{1}{4}}\|(V,E,\Psi)\|_{\mathcal{L}^{2}(\R^{3})}\|\Upsilon_{\epsilon}\|_{\mathcal{L}^{2}(\R^{3})}.
$$
By the Dominated convergence theorem, one obtain
$$
\int_{\R^{3}}\big(v(t),e(t),\psi(t)\big)\Upsilon dx\leq Ct^{\frac{1}{4}}\rightarrow0, \ \ \mbox{as}\ \ t\rightarrow0^{+}.
$$
Therefore,
$$\big(v(x,t),e(x,t),\psi(x,t)\big)\in BC_{w}([0,\infty);\mathcal{L}^{3,\infty}(\R^{3})),$$
together with the fact \eqref{E4.1}, we get $(u,b,\theta)\in BC_{w}([0,\infty);\mathcal{L}^{3,\infty}(\R^{3}))$.
Finally, we prove \eqref{T1}, by \eqref{E4.3}, we have
$$
\|u-e^{t\Delta}u_{0}\|_{p}=\Big\|\frac{1}{\sqrt{t}}V\Big(\frac{x}{\sqrt{t}}\Big)\Big\|_{p}\leq Ct^{\frac{3}{2p}-\frac{1}{2}},
$$
and by some readily calculation, we also get
$$
\|\nabla u-\nabla e^{t\Delta}u_{0}\|_{2}=\Big\|\frac{1}{\sqrt{t}}\nabla V\Big(\frac{x}{\sqrt{t}}\Big)\Big\|_{2}\leq Ct^{-\frac{1}{4}}.
$$
Similarly,
$$
\|b-e^{t\Delta}b_{0}\|_{p}\leq Ct^{\frac{3}{2p}-\frac{1}{2}},\ \ \|\theta-e^{t\Delta}\theta_{0}\|_{p}\leq Ct^{\frac{3}{2p}-\frac{1}{2}}
$$
and
$$
\|\nabla b-\nabla e^{t\Delta}b_{0}\|_{2}\leq Ct^{-\frac{1}{4}},\ \ \|\nabla\theta-\nabla e^{t\Delta}\theta_{0}\|_{2}\leq Ct^{-\frac{1}{4}}.
$$
Thus, we finish the proof of Theorem \ref{T1.1}.
\end{proof}

\setcounter{equation}{0}
 \setcounter{equation}{0}

\begin{appendices}
	\appendix
	\renewcommand{\appendixname}{}
	\renewcommand{\theequation}{\thesection \arabic{equation}}
	\setcounter{equation}{0}
	\section{The Lorentz space}
In this section, for the convenience of the reader, we provide the basic definition of Lorentz space and some frequently used results. Given a measurable subset $\Omega\subseteq\R^{n}$, and let $f$ be a measurable function defined on $\Omega$. The non-increasing rearrangement $f^{\ast}(t)$ of the function $f$ is defined as follows:
$$
f^{\ast}(t)=\inf\{\alpha>0:|\{x\in\Omega:|f(x)|>\alpha\}|\leq t\}
$$
where $|\{x\in\Omega:|f(x)|>\alpha\}|$ denote the measure of the set $\{x\in\Omega:|f(x)|>\alpha\}$. 
\begin{defn}\label{D-1}
	Let $1<p<\infty$ and $1\leq q\leq\infty$. The Lorentz space $L^{p,q}(\Omega)$ is the space of all measurable functions for which the following norm is finite:
	\begin{align*}
		\|f\|_{L^{p,q}(\Omega)}=\left \{
		\begin{array}{ll}
			\Big(\int_{0}^{\infty}\big(t^{1/p}f^{\ast}(t)\big)^{q}\frac{dt}{t}\Big)^{1/q}\ \ & {\rm if}\ \ 1< q<\infty,\\
			\sup_{t>0}t^{1/p}f^{\ast}(t)\ \ & {\rm if}\ \ q=\infty.
		\end{array}
		\right.
	\end{align*}
\end{defn}
In particular, when $q=\infty$, the Lorentz space $L^{p,\infty}$ is actually the weak-$L^{p}$ space, and its norm can be equivalently defined as follows:
$$
\|f\|_{L^{p,\infty}(\Omega)}=\sup_{\alpha>0}\alpha|\{x\in\Omega:|f(x)|>\alpha\}|^{1/q}.
$$
By the Definition \ref{D-1}, it follows directly that $L^{p,p}(\Omega)=L^{p}(\Omega)$. For $1<p<\infty$ and $1\leq q_{1}\leq q_{2}\leq\infty$, we have the following continuous embeddings:
\begin{equation}\label{embedd}
L^{p,q_{1}}(\Omega)\hookrightarrow L^{p,q_{2}}(\Omega),
\end{equation}
and the inclusion is known to be strict. It is well know that the function $|x|^{-n/p}$ is in $L^{p,\infty}(\R^{n})$ but not in $L^{p}(\R^{n})$.
Moreover,the Lorentz space can also be defined via interpolation theory(\cite{AF},\cite{LG1},\cite{LR}) as
\begin{equation*}\label{A-1}
L^{p,q}(\Omega)=\big(L^{p_{0}}(\Omega),L^{p_{1}}(\Omega)\big)_{\theta,q},
\end{equation*}
where $1<p_{0}<p<p_{1}<\infty$ and $\frac{1}{p}=\frac{1-\theta}{p_{0}}+\frac{\theta}{p_{1}}$ for some $\theta\in(0,1)$. This definition implies that $L^{p_{0}}(\Omega)\cap L^{p_{1}}(\Omega)$ is dense in $L^{p,q}(\Omega)$, and in particular, 
\begin{equation}\label{A-01}
L^{1}(\Omega)\cap L^{\infty}(\Omega)\  \mbox{is dense in} \  L^{p,q}(\Omega).
\end{equation}
Next, we recall the dual space of $L^{p,1}(\Omega)$, for $1<p<\infty$,
$$
\big(L^{p,1}(\Omega)\big)'=L^{p',\infty}(\Omega),\ \ p'=\frac{p}{p-1}.
$$
The definition of the dual space of a Lorentz space is very complicated, for other cases, see \cite{LG2} for details. We also recall a fundamental inequality in Lorentz Spaces called ‘Hunt's inequality’, which can be found in \cite{hunt}.
\begin{lem}\label{Holder}
	Suppose that $0<p,p_{1},p_{2}\leq\infty$ and $0<q,q_{1},q_{2}\leq\infty$ satisfies $\frac{1}{p}=\frac{1}{p_{1}}+\frac{1}{p_{2}}$ and $\frac{1}{q}=\frac{1}{q_{1}}+\frac{1}{q_{2}}$. Then the assumption that $f\in L^{p_{1},q_{1}}(\Omega)$ and $g\in L^{p_{2},q_{2}}(\Omega)$ implies that $fg\in L^{p,q}(\Omega)$, with the estimate
	$$
	\|fg\|_{L^{p,q}(\Omega)}\leq C(p_{1},p_{2},q_{1},q_{2})\|f\|_{L^{p_{1},q_{1}}(\R^{3})}\|g\|_{L^{p_{2},q_{2}}(\Omega)}.
	$$
\end{lem}
Next, we review the properties of convolution between Lorentz space.
\begin{lem}[\cite{LR}]\label{Lorentz}
	Let $1<p<\infty$, $1\leq q\leq\infty$, $\frac{1}{p'}+\frac{1}{p}=1$ and $\frac{1}{q'}+\frac{1}{q}=1$. Then convolution is a bounded bilinear operator:
	
	{\rm (i)} from $L^{p,q}(\R^{n})\times L^{1}(\R^{n})$ to $L^{p,q}(\R^{n})$;
	
	{\rm (ii)} from $L^{p,q}(\R^{n})\times L^{p',q'}(\R^{n})$ to $L^{\infty}(\R^{n})$;
	
	{\rm (iii)} from $L^{p,q}(\R^{n})\times L^{p_{1},q_{1}}(\R^{n})$ to $L^{p_{2},q_{2}}(\R^{n})$ for $1<p,p_{1},p_{2}<\infty$, $1\leq q,q_{1},q_{2}\leq\infty$, $\frac{1}{p_{2}}+1=\frac{1}{p}+\frac{1}{p_{1}}$ and $\frac{1}{q_{2}}=\frac{1}{q}+\frac{1}{q_{1}}$.
	
\end{lem}
The assumption on the initial values in Theorem \ref{T1.1} implies
$$
|(u_{0},b_{0},\theta_{0})|\leq\frac{C}{|x|}.
$$
Therefore, we conclude this paper by investigating some properties about the function $\frac{1}{|x|}$.
\begin{lem}\label{Ap1}
	Let $f(x)=\frac{1}{|x|}$ and $F(x,t)=e^{t\Delta}f(x)=G_{t}\ast f(x)$. Then we have
	\begin{equation}\label{A01}
	 F(x,t)\in BC_{w}([0,+\infty),L^{3,\infty}(\R^{3})).
		\end{equation}
	Furthermore, let $\bar{F}(x)=e^{\Delta/2}f(x)=G_{1/2}\ast f(x)$, 
	for all $x\in\R^{3}$, we also have
	\begin{align}\label{decay estimate1}
		\begin{split}
			|\bar{F}(x)|+|\nabla\bar{F}(x)|\leq C(1+|x|)^{-1},
		\end{split}
	\end{align}
	where the constant $C>0$ independent on $x$.
\end{lem}
\begin{proof}
	The proof of this result is fundamental, and for the convenience of the reader, we outline the proof.
	
Since $f(x)=\frac{1}{|x|}\in L^{3,\infty}(\R^{3})$.
By Lemma \ref{Lorentz}, we have
$$
\|F(t)\|_{L^{3,\infty}(\R^{3})}=\|G_{t}\ast f\|_{L^{3,\infty}(\R^{3})}\leq \|G_{t}\|_{1}\|f\|_{L^{3,\infty}(\R^{3})}< \infty.
$$
Next, we will show $F(x,t)$ is weak star continuous in $L^{3,\infty}(\R^{3})$ with respect to $t$. Indeed, it is sufficient to consider the continuity at $t=0$. That is, for any $\phi\in L^{\frac{3}{2},1}(\R^{3})$, we need to prove
\begin{equation}\label{A3-1}
|\langle F-f,\phi\rangle|=|\langle G_{t}\ast f-f,\phi\rangle|\rightarrow0, \ \ \ t\rightarrow0^{+}.
\end{equation}
Since $C_{0}^{\infty}(\R^{3})$ is dense in $L^{\frac{3}{2},1}(\R^{3})$, for any $\varepsilon>0$, one can find a sequence $\phi_{\varepsilon}\in C_{0}^{\infty}(\R^{3})$ such that
\begin{equation}\label{A3-001}
\|\phi-\phi_{\varepsilon}\|_{L^{\frac{3}{2},1}(\R^{3})}<\frac{\varepsilon}{3}.
\end{equation}
Let $v_{\varepsilon}(x,t)=G_{t}\ast\phi_{\varepsilon}$. Since $\phi_{\varepsilon}\in C_{0}^{\infty}(\R^{3})$ and using estimates for the parabolic equation, for any $k\geq0$, we can easily get that
$$
v_{\varepsilon}, \partial_{t}v_{\varepsilon}\in L^{\infty}\big((0,+\infty),W^{k,1}(\R^{3})\big).
$$
Then, applying the embedding theorem and \eqref{A-01} yields
\begin{equation}\label{A3-2}
v_{\varepsilon}\in C^{0,1}((0,+\infty),L^{\frac{3}{2},1}(\R^{3})).
\end{equation}
On the other hand, returning to \eqref{A3-1}, using Lemma \ref{Holder}, 
\begin{equation*}
	|\langle G_{t}\ast f-f,\phi\rangle|=|\langle f,G_{t}\ast\phi-\phi\rangle|\leq \|f\|_{L^{3,\infty}(\R^{3})}\|G_{t}\ast \phi-\phi\|_{L^{\frac{3}{2},1}(\R^{3})}.
\end{equation*}
As $t\rightarrow0^{+}$, by Lemma \ref{Lorentz}, \eqref{A3-001} and \eqref{A3-2}, it follows that
\begin{align*}
	\begin{split}
		\|&G_{t}\ast \phi-\phi\|_{L^{\frac{3}{2},1}(\R^{3})}\\&\leq \|G_{t}\ast(\phi-\phi_{\varepsilon})\|_{L^{\frac{3}{2},1}(\R^{3})}+
		\|\phi-\phi_{\varepsilon}\|_{L^{\frac{3}{2},1}(\R^{3})}+\|G_{t}\ast\phi_{\varepsilon}-\phi_{\varepsilon}\|_{L^{\frac{3}{2},1}(\R^{3})}\\&
		\leq \frac{\varepsilon}{3}+\frac{\varepsilon}{3}+\frac{\varepsilon}{3}< \varepsilon.
	\end{split}
\end{align*}
Thus, the arbitrariness of $\varepsilon$ allows us to obtain \eqref{A01}.

In the following, we prove \eqref{decay estimate1}. By the well-known decay estimate of the heat kernel, one can get $G_{1/2},\nabla G_{1/2}$ belong to $L^{p,q}(\R^{3})$ (for any $1<p<\infty$ and $1\leq q\leq\infty$). On the one hand, using Lemma \ref{Lorentz}, we get
\begin{align}\label{A3-3}
	\begin{split}
	\|\bar{F}(x)\|_{\infty}+\|\nabla\bar{F}(x)\|_{\infty}&\leq C\Big(\|G_{1/2}\|_{L^{\frac{3}{2},1}(\R^{3})}+\|\nabla G_{1/2}\|_{L^{\frac{3}{2},1}(\R^{3})}\Big)\|f(x)\|_{L^{3,\infty}(\R^{3})}\\& \leq C.
	\end{split}
\end{align}
On the other hand, 
\begin{align}\label{A3-4}
	\begin{split}
		\Big||x|\bar{F}(x)\Big|&=\Big|\frac{|x|}{(2\pi)^{\frac{3}{2}}}\int_{\R^{3}}e^{-\frac{|x-y|^{2}}{2}}f(y)dy\Big|\\&\leq
		\frac{1}{(2\pi)^{\frac{3}{2}}}\int_{\R^{3}}|x-y|e^{-\frac{|x-y|^{2}}{2}}|f(y)|dy+\frac{1}{(2\pi)^{\frac{3}{2}}}\int_{\R^{3}}e^{-\frac{|x-y|^{2}}{2}}|y||f(y)|dy
		\\&\leq \big\||\cdot|G_{1/2}(x)\big\|_{L^{\frac{3}{2},1}(\R^{3})}\|f(x)\|_{L^{3,\infty}(\R^{3})}+1\\&
		\leq C.
	\end{split}
\end{align}
Similarly,
\begin{align}\label{A3-5}
	\begin{split}
		\Big||x|\nabla\bar{F}(x)\Big|&= \Big|\frac{|x|}{(2\pi)^{\frac{3}{2}}}\nabla\int_{\R^{3}}e^{-\frac{|x-y|^{2}}{2}}f(y)dy\Big|
		\\&\leq \big\||\cdot|\nabla G_{1/2}(x)\big\|_{L^{\frac{3}{2},1}(\R^{3})}\|f(x)\|_{L^{3,\infty}(\R^{3})}+\|\nabla G_{1/2}(x)\|_{1}\\&\leq C.
	\end{split}
\end{align}
The above calculation depends on the rapid decay of $G_{1/2}$ and $\nabla G_{1/2}$ at the spatial infinity. Collecting \eqref{A3-3}-\eqref{A3-5}, we conclude the proof of \eqref{decay estimate1}.

\end{proof}

\end{appendices}

\bibliographystyle{unsrt}

 \end{CJK}
 \end{document}